\documentclass[review,12pt]{elsarticle}
\usepackage{amsmath}
\usepackage{amsfonts}
\usepackage{geometry}   
\usepackage[parfill]{parskip}    
\usepackage{graphicx}
\usepackage{amssymb}
\usepackage{epstopdf}
\usepackage{psfrag}

\newcommand {\N}{\mathbb{N}}

\newcommand {\R}{\mathbb{R}}

\newcommand {\bea}{\begin{eqnarray}}
\newcommand {\ea}{\end{eqnarray}}

\newtheorem{remark}{Remark}[section]

\newenvironment{proof}[1][Proof]{\textbf{#1.} }{\hspace{\stretch{1}}\rule{0.5em}{0.5em}}

\usepackage{xspace}

\usepackage{subfigure}

\newtheorem{dfn}{Definition}
\newtheorem{thm}{Theorem}
\newtheorem{lem}{Lemma}

\newtheorem{rem}{Remark}
\newtheorem{prp}{Proposition}

\newtheorem{ass}{Assumption}

\newcommand{\assref}[1]{{Assumption~\ref{#1}}}
\newcommand{\propref}[1]{{Proposition~\ref{#1}}}

\journal{Indagationes Mathematicae}
\begin{document}
\begin{frontmatter}

\title{Strong convergence of an fractional exponential integrator scheme  for the finite element discretization of time-fractional SPDE driven by standard and fractional Brownian motions}
\author[ajn]{Aurelien Junior Noupelah}
\ead{noupsjunior@yahoo.fr}
\address[ajn]{Department of Mathematics and Computer Sciences, University of Dschang, P.O. BOX 67, Dschang, Cameroon}

\author[at,atb]{Antoine Tambue}
\cortext[cor1]{Corresponding author}
\ead{antonio@aims.ac.za}
\address[at]{Department of Computer science, Electrical engineering and Mathematical sciences,  Western Norway University of Applied Sciences, Inndalsveien 28, 5063 Bergen.}

\author[ajn]{Jean Louis Woukeng}
\ead{jwoukeng@yahoo.fr}

\begin{abstract}
The aim of this work is to provide  the first strong convergence result of numerical approximation of a general time-fractional second order stochastic   partial differential equation involving a Caputo derivative in time of order $\alpha\in(\frac 12; 1)$ and  driven simultaneously by a multiplicative standard Brownian motion and additive fBm with Hurst parameter $H\in(\frac 12, 1)$, more realistic to model the random effects on transport of particles in medium with thermal memory. We prove  the existence and uniqueness results and perform  the spatial discretization using the finite element and the temporal discretization using a fractional exponential integrator scheme. We provide the temporal and spatial convergence proofs for our fully discrete scheme and the result shows that the convergence orders depend on the regularity of the initial data, the power of the fractional derivative, and the Hurst parameter $H$.  
\end{abstract}
\end{frontmatter}

	\section{Introduction}
\label{intro}
We analyse the strong numerical approximation of the following time fractional SPDE with initial value of the following type
\begin{equation}
	\label{pbfrac}
	\left\{
	\begin{array}{ll}
		^C\partial^{\alpha}_tX(t)+AX(t)=F(X(t))+I_t^{1-\alpha}\left[G(X)\frac{dW(t)}{dt}+\Phi\frac{dB^H(t)}{dt}\right], \\
		X(0)=X_0,\,\,\,\,t\in[0,T].
	\end{array}
	\right.
\end{equation} 
On the Hilbert space $\mathcal{H}=L^2(\Lambda)$, where $\Lambda\subset\R^d$, $d=1,2,3$ is bounded and has smooth boundary, $T>0$ is the final time, $A$ is a linear operator which is unbounded, not necessarily self-adjoint and is assumed to generate a semigroup $S(t):=e^{-tA}$,  $^C\partial^{\alpha}_t$ is the Caputo fractional derivative with $\alpha\in(\frac 12,1)$ and $I^{1-\alpha}_t$ is the fractional integral operator which will be given in the next section. The functions $F:\mathcal{H}\rightarrow \mathcal{H}$, $G: \mathcal{H}\rightarrow \mathcal{H}$ and $\Phi: \mathcal{H}\rightarrow\mathcal{H}$ are deterministic mappings that will be specified later, $X_0$ is the initial data which is random, $W(t)=W(x,t)$ is a $\mathcal{H}$-valued $Q$-Wiener process defined in a filtered probability space $(\Omega, \mathcal{F}, \mathbb{P},\{\mathcal{F}_t\}_{t\geq 0})$ and the term $B^H(t) = B^H(x, t)$ in \eqref{pbfrac} is a $\mathcal{H}$-valued fractional $Q_1$-Brownian motion with Hurst parameter $H\in(\frac 12, 1)$ defined in a filtered probability space $(\Omega, \mathcal{F}, \mathbb{P},\{\mathcal{F}_t\}_{t\geq 0})$, where the covariances operators $Q: \mathcal{H}\rightarrow \mathcal{H}$  and  $Q_1: \mathcal{H}\rightarrow \mathcal{H}$ are positive and linear self-adjoint operators. The filtered probability space $(\Omega, \mathcal{F}, \mathbb{P},\{\mathcal{F}_t\}_{t\geq 0})$ is assumed to fulfil the usual assumptions (see \cite[Definition 2.1.11]{Pre} ). It is well known \cite{Car, Pra} that the noises can be represented as follows
\begin{eqnarray}
	\label{def_Brow}
	W(x,t)=\sum_{i=0}^{\infty}\beta_i(t)Q^{\frac 12}e_i(x)=\sum_{i=0}^{\infty}\sqrt{q_i}\beta_i(t)e_i(x),\\
	\label{def_fracBrow}
	B^H(x,t)=\sum_{i=0}^{\infty}\beta_i^H(t)Q_{1}^{\frac 12}e_i^1(x)=\sum_{i=0}^{\infty}\sqrt{q_i^1}\beta_i^H(t)e_i^1(x),
\end{eqnarray}
where $q_i,\,e_i,\,i\in\N$ are respectively the eigenvalues and eigenfunctions of the covariance operator $Q$,  $q_i^1,\,e_i^1,\,i\in\N$ are respectively the eigenvalues and eigenfunctions of the covariance operator $Q_1$,  $\beta_i$ are mutually independent and identically distributed standard normal distributions and $\beta_i^H$ are mutually independent and identically distributed fractional Brownian motion (fBm). The noises $W$ and $B^H$ are supposed to be independent. Precise assumptions on the nonlinear mappings $G$ and $\Phi$ to ensure the existence of the mild solution of \eqref{pbfrac} will be given in the following section.

Equation of type \eqref{pbfrac} with $\Phi = 0$ might be used to model random effects on transport of particles in medium with thermal memory \cite{Zouc}. So due to the self-similar and long-range dependence properties of the fBm, when modelling such phenomena, it is recommended to incorporate the fBm process in order to obtain a more realistic model. During the last few decades, the theory of fractional partial differential equations has gained considerable interest over time. From the point of view of computations, several numerical methods have been proposed for solving time fractional partial differential equations (for details, see \cite{Jia,Liu,For,Pri,Gao,Elz} and the reference therein). Note that the time stepping methods used in all the works mentioned until now  are based on finite difference methods. However theses schemes are explicit, but unstable, unless the time stepsize is very small.  To solve that drawback, numerical method based on exponential integrators of Adams type have  been  proposed in \cite{Gar1}. The price to pay is the computation of Mittag-Leffler (ML) matrix functions. As ML matrix function is the generalized form of the exponential of matrix function, works in \cite{Gar2,Mor,Pop} have extended  some exponential  computational techniques to ML. Note that up to  now all the numerical algorithms presented are for  time fractional deterministic PDEs with self adjoint  linear operators. 

Actually, few works have been done for numerical methods for Gaussian noise for time fractional stochastic partial differential equation (see \cite{Gun,Zoub,Zouc,Zoud}).
Note that this above works have been done for  self adjoint linear operator, so numerical study for \eqref{pbfrac} with $\Phi \neq 0$ and non self adjoint operator is still an open problem in the field, to the best of our knowledge. 
However, it is important to  mention that  if $H\neq \frac 12$ the process $B^H$ is not a semi-martingale  and the standard stochastic calculus techniques are therefore obsolete while studying SPDEs of type \eqref{pbfrac}.  Alternative approaches to the standard It\^o calculus are therefore required in order to build a stochastic calculus framework for such fBm.   In recent years, there have been various developments of stochastic calculus and stochastic differential equations with respect to the fBm especially for $H\in(\frac 12, 1)$ (see, for example \cite{Alo,Car,Mis}) and  theory of SPDEs driven by fractional Brownian motion  has been also studied. For example, linear and semilinear stochastic equations in a Hilbert space with an infinite dimensional fractional Brownian motion are considered in \cite{Duna,Dunb}. However numerical scheme for time fractional   SPDEs  \eqref{pbfrac} driven  both by fractional Brownian motion and standard Brownian motion have been lacked in the scientific literature  to the best of our knowledge. 
Our goal in this work is to build  the first  numerical method to approximate the time fractional stochastic partial differential equation  \eqref{pbfrac} driven simultaneously by a multiplicative standard Brownian motion and an additive fractional Brownian motion with Hurst parameter $H\in(\frac 12, 1)$ using  finite element method for spatial approximation and a fractional version of exponential \cite{Lor,Nouc} Euler scheme for time discretization.  Since the ML function is more challenging than the exponential function and 
our  main result is based  on novel preliminary  results on  ML functions.  The analysis is complicated and is very different to that of  a standard  exponential integrator scheme \cite{Nou}(where $\alpha=1$)  since the fractional derivative is not local and therefore  numerical solution at given time  depends to all previous numerical solutions up to that time. This is   in contrast to the  standard  exponential  numerical scheme  where the  numerical solution  at a given time depends only of that of the previous nearest solution.
We provided the strong convergence of our fully discrete scheme for \eqref{pbfrac}. Our strong convergence results examine how the convergence orders depend on the regularity of the initial data, the power of the fractional derivative, and the Hurst parameter.   

The rest of the paper is structured as follows. In Section \ref{mathsetting}, Mathematical settings for standard and fractional calculus are presented, along with the existence, uniqueness, and regularities results of the mild solution of SPDE \eqref{pbfrac}. In Section \ref{regpaper4}, numerical schemes for SPDE \eqref{pbfrac} are presented, we discuss about space and time regularity of the mild solution $X(t)$ of \eqref{pbfrac} given by \eqref{solpbfrac}.  The spatial error analysis  is done in Section \ref{spa_convpaper4}. We end the paper in Section \ref{schemesfrac}, by presenting  the strong convergence proof of the our novel  numerical scheme.

\section{Mathematical setting, main assumptions and well posedness problem}
\label{mathsetting}
In this section, we review briefly some useful results on standard and fractional calculus, introduce notations, definitions, preliminaries results which will be needed throughout this paper and the proof of existence and uniqueness of the mild solution of \eqref{pbfrac}.
\begin{dfn}\textbf{[Fractional Brownian motion]}\cite{Mis,Nou}
	\label{fBm}
	An $\mathcal{H}$-valued Gaussian process $\{B^H(t),t\in[0,T]\}$ on $(\Omega, \mathcal{F}, \mathbb{P},\{\mathcal{F}_t\}_{t\geq 0})$ is called a fractional Brownian motion with Hurst parameter $H\in(0,1)$ if 
	\begin{itemize}
		\item $\mathbb{E}[B^H(t)]=0$ for all $t\in\R$,
		\item $\text{Cov}(B^H(t),B^H(s))=\frac 12\left(|t|^{2H}+|s|^{2H}-|t-s|^{2H}\right)$ for all  $t,s\in\R$,
		\item $\{B^H(t),t\in[0,T]\}$ has continuous sample paths  $\mathbb{P}$ a.s.,
	\end{itemize} 
	where $\text{Cov}(X,Y)$ denotes the covariance operator for the Gaussian random variables $X$ and $Y$ and $\mathbb{E}$ stands for the mathematical expectation on $(\Omega, \mathcal{F}, \mathbb{P},\{\mathcal{F}_t\}_{t\geq 0})$. 
\end{dfn}  
Notice that if $H=\frac 12$, the fractional Brownian motion coincides with the standard Brownian motion. Throughout this paper the Hurst parameter $H$ is assumed to belong to $(\frac 12,1)$.

Let $\left(K,\langle .,.\rangle_K,\|.\|\right)$ be a separable Hilbert space. For $p\geq 2$ and for a Banach space U, we denote by $L^p(\Omega,U)$ the Banach space of $p$-integrable $U$-valued random variables. We denote by $L(U,K)$ the space of bounded linear mapping from $U$ to $K$ endowed with the usual operator norm $\|.\|_{L(U,K)}$ and $\mathcal{L}_2(U,K)=HS(U,K)$ the space of Hilbert-Schmidt operators from $U$ to $K$ equipped  with the following norm
\begin{eqnarray}
	\label{normL2UK}
	\left\|l\right\|_{\mathcal{L}_2(U,K)}:=\left(\sum_{i=0}^{\infty}\|l\psi_i\|^2\right)^{\frac 12},\,\,\,\,\,l\in \mathcal{L}_2(U,K),
\end{eqnarray}
where $(\psi_i)_{i\in\N}$ is an orthonormal basis on $U$. The sum in $\eqref{normL2UK}$ is independent of the choice of the orthonormal basis of $U$. We use the notation $L(U,U)=:L(U)$ and $\mathcal{L}_2(U,U)=:\mathcal{L}_2(U)$. It is well known that for all $l\in L(U,K)$ and $l_1\in \mathcal{L}_2(U)$, $ll_1\in \mathcal{L}_2(U,K)$ and 
\begin{eqnarray*}
	\|ll_1\|_{\mathcal{L}_2(U,K)}\leq \left\|l\right\|_{L(U,K)}\|l_1\|_{\mathcal{L}_2(U)}.
\end{eqnarray*} 
We denote by $L^0_2:=HS(Q^{\frac 12}(H),H)$ the space of Hilbert-Schmidt operators from $Q^{\frac 12}(\mathcal{H})$  to $\mathcal{H}$ with corresponding norm $\|.\|_{L^0_2}$ defined by 
\begin{eqnarray}
	\label{normL02}
	\|l\|_{L_2^0}:=\left\|lQ^{\frac 12}\right\|_{HS}=\left(\sum_{i\in\N}\|lQ^{\frac 12}e_i\|^2\right)^{\frac 12},\,\,\,\,\,l\in L_2^0,
\end{eqnarray}
where $(e_i)_{i=0}^{\infty}$ are orthonormal basis of $\mathcal{H}$. 
The following lemma will be very important throughout this paper.
\begin{lem}(It\^{o} Isometry: \cite[(4.30)]{Pra}, \cite[(12)]{Nou})
	\begin{enumerate}
		\item [(i)] Let $\theta\in L^2([0,T];L^0_2)$, then the following holds
		\begin{eqnarray}
			\label{browint}
			\mathbb{E}\left[\left\|\int_0^T\theta(s)dW(s)\right\|^2\right]=\mathbb{E}\left[\int_0^T\left\|\theta(s)\right\|^2_{L^0_2}ds\right].
		\end{eqnarray}
		\item [(ii)] Let $\Phi: \mathcal{H}\rightarrow \mathcal{H}$, then the following holds
		\begin{eqnarray}
			\label{fracint}
			\mathbb{E}\left[\left\|\int_0^T\Phi dB^H(s)\right\|^2\right]\leq C(H)\sum_{i=0}^{\infty}\left(\int_{0}^T\left\|\Phi Q^{\frac 12}e_i\right\|^{\frac 1H}ds\right)^{2H}.
		\end{eqnarray}
	\end{enumerate}
\end{lem}
\begin{rem}
	Note that in the case $H=\frac 12$, the constant $C(H)$ in \eqref{fracint} is 1 and the inequality becomes the equality. In this case, the result \eqref{fracint} is then identically to \eqref{browint}.
\end{rem}
More details on the definition of stochastic integral with respect to fractional $Q$-Brownian motion and their property are given in e.g \cite{Alo, Duna,Dunb,Mis}
\begin{dfn}(\cite[(2.1.1),(2.4.17)]{Kil})
	\label{cap}
	The Caputo-type derivative of order $\alpha$ with respect to $t$ is defined for all $t>0$ by
	\begin{eqnarray}
		\label{cap_der}
		^C\partial^{\alpha}_tX(t)=	
		\left\{
		\begin{array}{ll}
			\frac{1}{\Gamma(1-\alpha)}\int_{0}^t\frac{\partial X(s)}{\partial s}\frac{ds}{(t-s)^{\alpha}},\,\,\,\,0<\alpha<1 \\
			\frac{\partial X}{\partial t},\,\,\,\,\alpha=1,
		\end{array}
		\right.
	\end{eqnarray}
	and the Riemann-Liouville fractional integral operator $I^{\alpha}_t$ is defined for all $t>0$ by
	\begin{eqnarray}
		\label{RieLiou_int}
		I^{\alpha}_tX(t)=	
		\left\{
		\begin{array}{ll}
			\frac{1}{\Gamma(\alpha)}\int_{0}^t(t-s)^{\alpha-1}X(s)ds,\,\,\,\,0<\alpha<1 \\
			X(t),\,\,\,\,\alpha=0.
		\end{array}
		\right.
	\end{eqnarray} 
	where $\Gamma(\cdot)$ is the gamma function.
\end{dfn}
\begin{prp}
	\label{Mai}
	Considering the generalized Mittag-Leffler function (see \cite{Hau}) $E_{\alpha,\beta}(t)$ and let the Mainardi's Wright-type function (see \cite{Mai}) $M_{\alpha}(\theta)$ defined as follows:
	\begin{eqnarray*}
		E_{\alpha,\beta}(t)=\sum_{k=0}^{\infty}\frac{t^k}{\alpha k+\beta}\hspace{1cm}\text{and}\hspace{1cm}M_{\alpha}(\theta)=\sum_{n=0}^{\infty}\frac{(-1)^n\theta^n}{n!\Gamma(1-\alpha(1+\theta))},\,\,\,\,\,0<\alpha<1,\,\,\,\theta>0,
	\end{eqnarray*}
	then the following results hold:
	\begin{eqnarray}
		\label{Mai1}
		M_{\alpha}(\theta)\geq 0,\hspace{1cm}\int_{0}^{\infty}\theta^{\mu}M_{\alpha}(\theta)d\theta=\frac{\Gamma(1+\mu)}{\Gamma(1+\alpha\mu)},\hspace{0.5cm}-1<\mu<\infty,\hspace{0.5cm}\theta>0
	\end{eqnarray}
	and
	\begin{eqnarray}
		\label{Mai2}
		E_{\alpha,1}(t)=\int_0^{\infty}M_{\alpha}(\theta)e^{t\theta}d\theta,\,\,\,\,\,E_{\alpha,\alpha}(t)=\int_0^{\infty}\alpha\theta M_{\alpha}(\theta)e^{t\theta}d\theta.
	\end{eqnarray}
\end{prp}
In the rest of this paper to simplify the presentation, we assume the SPDE \eqref{pbfrac} to be second order of the following type.
\begin{eqnarray}
	\label{modelfrac}
	&&^C\partial_t^{\alpha}X(t)+\left[-\nabla\cdot(D\nabla X(t,x))+q\cdot \nabla X(t,x)\right]\nonumber\\
	&=&f(x,X(t,x))+I_t^{1-\alpha}\left[g(x,X(t,x))\frac{dW(t,x)}{dt}+\phi(x)\frac{dB^H(t,x)}{dt}\right]
\end{eqnarray}
where $f, g:\Lambda\times\R\rightarrow \R$ is globally Lipschitz continuous,  $\phi:\R\rightarrow \R$ is bounded. In the abstract framework \eqref{modelfrac}, the linear operator $A$ takes the form 
\begin{eqnarray*}
	&&Au=-\sum_{i,j=1}^d\frac{\partial}{\partial x_i}\left(D_{i,j}(x)\frac{\partial u}{\partial x_j}\right)+\sum_{i=1}^d q_i(x)\frac{\partial u}{\partial x_i}\\
	&&D=\left(D_{i,j}\right)_{1\leq i,j\leq d},\hspace{1cm}q=\left(q_i\right)_{1\leq i\leq d},
\end{eqnarray*}
where $D_{i,j}\in L^{\infty}(\Lambda)$, $q_i\in L^{\infty}(\Lambda)$. We assume that there exists a positive constant $c_1>0$ such that
\begin{eqnarray*}
	\sum_{i,j=1}^d D_{i,j}(x)\xi_i\xi_j\geq c_1|\xi|^2,\hspace{1cm}\xi\in\R^d,\hspace{0.5cm}x\in\bar{\Omega},
\end{eqnarray*}
The functions $F:\mathcal{H}\rightarrow \mathcal{H}$, $G:\mathcal{H}\rightarrow L^0_2$ and $\Phi \in L^0_2$ are defined by 
\begin{eqnarray*}
	(F(v))(x)=f(x,v(x)),\quad (G(v)u)(x)=g(x,v(x))\cdot u(x),\quad (\Phi w)(x)=\phi(x)\cdot w(x)
\end{eqnarray*}
for all $x\in\Lambda$, $v\in \mathcal{H}$, $u,\,w\in Q^{1/2}(\mathcal{H})$. As in \cite{Fuj,Lor}, we introduce two spaces $\mathbb{H}$, and $V$ such that $\mathbb{H}\subset V$; the two spaces depend on the boundary conditions of $\Lambda$ and the domain of the operator A. For Dirichlet (or first-type) boundary conditions, we take 
\begin{equation*}
	V=\mathbb{H}=H^1_0(\Lambda)=\{v\in H^1(\Lambda):v=0\,\,\text{on}\,\,\partial \Lambda\}.
\end{equation*}
For Robin (third-type) boundary condition and Neumann (second-type) boundary condition, which is a special case of Robin boundary condition, we take $V=H^1(\Omega)$ 
\begin{equation*}
	\mathbb{H}=\{v\in H^2(\Lambda): \partial v/\partial v_{\mathcal{A}}+\alpha_0 v=0,\hspace{0.5cm}\text{on} \hspace{0.5cm}\partial\Lambda\},\hspace{1cm}\alpha_0\in\R,
\end{equation*}
where $\partial v/\partial v_{\mathcal{A}}$ is the normal derivative of $v$ and $v_{\mathcal{A}}$ is the exterior pointing normal $n=(n_i)$ to the boundary of $\mathcal{A}$ given by
\begin{equation*}
	\partial v/\partial v_{\mathcal{A}}=\sum_{i,j=1}^d n_i(x)D_{i,j}(x)\frac{\partial v}{\partial x_j},\,\,\,\,\,x\in\partial\Lambda.
\end{equation*}
Using  G\aa rding's inequality (see e.g. \cite{ATthesis}), it holds that there exist two constants $c_0$ and $\lambda_0>0$ such that   the bilinear form $a(.,.)$ associated to $A$ satisfies
\begin{eqnarray}
	\label{ellip1}
	a(v,v)\geq \lambda_0\Vert v \Vert^2_{H^1(\Lambda)}-c_0\Vert v\Vert^2, \quad v\in V.
\end{eqnarray}
By adding and subtracting $c_0Xdt$ in both sides of \eqref{modelfrac}, we have a new linear operator still denoted by $A$, and the corresponding  bilinear form is also still denoted by $a$. Therefore, the following coercivity property holds
\begin{eqnarray}
	\label{ellip2}
	a(v,v)\geq \lambda_0\Vert v\Vert^2_{H^1(\Lambda)},\quad v\in V.
\end{eqnarray}
Note that we have create a new linear term $-c_0X$ in the right-hand side of \eqref{modelfrac}. Thus we obtain a new equivalent form to \eqref{modelfrac} as
\begin{eqnarray}
	\label{modelfrac2}
	^C\partial_t^{\alpha}X(t)+\left[-\nabla\cdot(D\nabla X(t,x))+q\cdot \nabla X(t,x)-c_0X(t,x)\right]\nonumber\\
	=f(x,X(t,x))-c_0X(t,x)+I_t^{1-\alpha}\left[g(x,X(t,x))\frac{dW(t,x)}{dt}+\phi(x)\frac{dB^H(t,x)}{dt}\right].
\end{eqnarray} 
we rewrite it in its contracted form as follows
\begin{eqnarray}
	\label{pbfrac2}
	^C\partial_t^{\alpha}X(t)+A X(t)=F(X(t))+I_t^{1-\alpha}\left[G(X)\frac{dW(t)}{dt}+\Phi\frac{dB^H(t)}{dt}\right].
\end{eqnarray}
Note that the expression of nonlinear term $F$ has changed as we included the term $-c_0X$ in the new nonlinear term that we still denote by $F$. The coercivity property \eqref{ellip2} implies that  $A$ is the infinitesimal generator of a \emph{contraction} semigroup $S(t)=e^{-t A}$  on $L^{2}(\Lambda)$ \cite{Fuj}.  Note that this is due to the fact that 
the real part of the eigenvalues of  $A$  is positive.
Note also  that the coercivity  property \eqref{ellip2} also implies that $A$ is a
positive operator and its fractional powers are well defined and  for any $\alpha>0$ we have
\begin{eqnarray}
	\label{fractional}
	\left\{\begin{array}{rcl}
		A^{-\alpha} & =& \frac{1}{\Gamma(\alpha)}\displaystyle\int_0^\infty  t^{\alpha-1}{\rm e}^{-tA}dt,\\
		A^{\alpha} & = & (A^{-\alpha})^{-1},
	\end{array}\right.
\end{eqnarray}
where $\Gamma(\alpha)$ is the Gamma function.
\begin{remark}
	\label{contraction}
	As we have mentioned, the linear  operator  $A$ is the infinitesimal generator of a contraction semigroup $S(t)=e^{-t A}$,  and therefore
	\begin{eqnarray}
		\label{contract}
		\Vert S(t)\Vert_{L(\mathcal{H})} \leq 1.
	\end{eqnarray}
	Let 
	\begin{eqnarray}
		\label{semi}
		\mathcal{S}_1(t):=E_{\alpha,1}(-t^{\alpha}A)\hspace{1cm}\text{and}\hspace{1cm} \mathcal{S}_2(t):=E_{\alpha,\alpha}(-t^{\alpha}A),
	\end{eqnarray}
	 be the fractional semigroups  from \propref{Mai} and \eqref{contract}, we also have
	\begin{eqnarray}
		\label{contract1}
		\Vert \mathcal{S}_1(t)\Vert_{L(\mathcal{H})} \leq 1 \hspace{1cm}\text{and}\hspace{1cm}\Vert \mathcal{S}_2(t)\Vert_{L(\mathcal{H})} \leq \frac{\alpha\Gamma(2)}{\Gamma(1+\alpha)}.
	\end{eqnarray}
	Note that  \eqref{contract}-\eqref{contract1} are also hold is  $S(t)$ and $S_1(t)$ are replaced respectively  by their semi discrete form $S_h(t)$ and $S_{1h}(t)$ obtained after  the finite element method.
	In the sequel of this paper, the contraction propriety \eqref{contract}-\eqref{contract1}  for  $S_{1h}(t)$ will play a key role in the analysis our numerical scheme.
\end{remark}

Following the same lines as \cite[(2.2)-(2.5)]{Zoub,Zouc} and using the equivalent model \eqref{pbfrac2}, we represent the mild solution of \eqref{pbfrac} as:
\begin{dfn}
	\label{mildsolpbfrac}
	For any $0<\alpha<1$, a stochastic process $\{X(t), t\in[0,T]\}$ is called mild solution of \eqref{pbfrac} if 
	\begin{enumerate}
		\item [1.] $X(t)$ is $\mathcal{F}_t$-adapted on the filtration $(\Omega, \mathcal{F}, \mathbb{P},\{\mathcal{F}_t\}_{t\geq 0})$,
		\item [2.] $\{X(t), t\in[0,T]\}$ is measurable and $\mathbb{E}\left[\int_{0}^T\left\|X(t)\right\|^2dt\right]<\infty$,
		\item [3.] For all $t\in[0,T]$, 
		\begin{eqnarray}
			\label{solpbfrac}
			X(t)&=&\mathcal{S}_1(t)X_0+\int_0^t (t-s)^{\alpha-1}\mathcal{S}_2(t-s)F(X(s))ds\nonumber\\
			&+&\int_{0}^{t}\mathcal{S}_1(t-s)G(X(s))dW(s)+\int_{0}^{t}\mathcal{S}_1(t-s)\Phi dB^H(s),
		\end{eqnarray}
		hold a.s. where $\mathcal{S}_1(t)$ and $\mathcal{S}_2(t)$ are defined by \eqref{semi}.
	\end{enumerate}
\end{dfn}
Thanks to \eqref{Mai2}, the fractional semigroup operators $\mathcal{S}_1(t)$ and $\mathcal{S}_2(t)$ can be rewritten as 
\begin{eqnarray}
	\label{S1}
	\mathcal{S}_1(t):=E_{\alpha,1}(-At^{\alpha})=\int_0^{\infty}M_{\alpha}(\theta)e^{-A\theta t^{\alpha}}d\theta=\int_0^{\infty}M_{\alpha}(\theta)S(\theta t^{\alpha})d\theta,
\end{eqnarray}
and
\begin{eqnarray}
	\label{S2}
	\mathcal{S}_2(t):=E_{\alpha,\alpha}(-At^{\alpha})=\int_0^{\infty}\alpha\theta M_{\alpha}(\theta)e^{-A\theta t^{\alpha}}d\theta=\int_0^{\infty}\alpha\theta M_{\alpha}(\theta)S(\theta t^{\alpha})d\theta.
\end{eqnarray}
we obtain the following properties of that fractional semigroup $(\mathcal{S}_i(t))_{t\in(0,T)}$, $i=1, 2$.  
\begin{lem}(\cite[Lemma 4]{Nouc})
	\label{fracbound}
	Let $t\in(0,T)$, $0<t_1<t_2\leq T$, $T<\infty$,  $\frac 12<\alpha<1$, $\rho\geq 0$, $0\leq \eta<1$, $0\leq \kappa<\varpi\leq 1$ and  $\delta\geq 0$, there exists a constant $C>0$ such that for all $i=1,2$
	\begin{eqnarray}
		\label{semigroup_prp}
		\|A^{\rho}\mathcal{S}_i(t)\|_{L(\mathcal{H})}\leq C t^{-\alpha\rho},\hspace{1cm}\|A^{-\eta}(\mathcal{S}_1(t_2)-\mathcal{S}_1(t_1))\|_{L(\mathcal{H})}\leq C (t_2-t_1)^{\alpha\eta},
	\end{eqnarray}
	\begin{eqnarray}
		\label{semigroup_prp1*}
		\|A^{\kappa}(\mathcal{S}_1(t_2)-\mathcal{S}_1(t_1))\|_{L(\mathcal{H})}\leq C (t_2-t_1)^{\alpha(\varpi-\kappa)}t_1^{-\alpha\varpi},
	\end{eqnarray}
	\begin{eqnarray}
		\label{semigroup_prp2*}
		\left\|t_1^{\alpha-1}\mathcal{S}_2(t_1)-t_2^{\alpha-1}\mathcal{S}_2(t_2)\right\|_{L(\mathcal{H})}\leq C (t_2-t_1)^{1-\alpha}t_1^{\alpha-1}t_2^{\alpha-1},
	\end{eqnarray}
	and
	\begin{eqnarray}
		\label{semigroup_prp*}
		A^{\delta}\mathcal{S}_i(t)=\mathcal{S}_i(t)A^{\delta}\hspace{0.25cm}\text{on}\hspace{0.25cm}D(A^{\delta}).
	\end{eqnarray}
\end{lem}
\begin{proof}
	See \cite[Lemma 4]{Nouc} for the proof of \eqref{semigroup_prp}, \eqref{semigroup_prp*} and \eqref{semigroup_prp2*}. The proof of \eqref{semigroup_prp1*} is similar to that of \cite[(29)]{Nouc}, we have hence
	\begin{eqnarray*}
		&&\left\|A^{\kappa}\left(\mathcal{S}_1(t_2)-\mathcal{S}_1(t_1)\right)\right\|_{L(\mathcal{H})}\\
		&=&\left\|\int_{0}^{\infty}A^{\kappa}M_{\alpha}(\theta)\left(S(\theta t_2^{\alpha})-S(\theta t_1^{\alpha})\right)d\theta\right\|_{L(\mathcal{H})}\\
		&\leq&\int_{0}^{\infty}M_{\alpha}(\theta)\left\|A^{-\varpi}S(\theta t_1^{\alpha})\right\|_{L(\mathcal{H})}\left\|A^{\kappa-\varpi}\left(e^{-A\theta (t_2^{\alpha}-t_1^{\alpha})}-I\right)\right\|_{L(\mathcal{H})}d\theta\\
		&\leq&C\int_{0}^{\infty}\left[\theta\,(t_2^{\alpha}-t_1^{\alpha})\right]^{\varpi-\kappa}\left(\theta t_1^{\alpha}\right)^{-\varpi} M_{\alpha}(\theta)d\theta\\
		&\leq& C(t_2^{\alpha}-t_1^{\alpha})^{\varpi-\kappa}t_1^{-\alpha\varpi}\int_{0}^{\infty}\theta^{-\kappa} M_{\alpha}(\theta)d\theta\\
		&\leq& C\frac{\Gamma(1-\kappa)}{\Gamma(1-\alpha\kappa)}(t_2-t_1)^{\alpha(\varpi-\kappa)}t_1^{-\alpha\varpi}\\
		&\leq&  C\,(t_2-t_1)^{\alpha(\varpi-\kappa)}t_1^{-\alpha\varpi}.
	\end{eqnarray*}
\end{proof}
\begin{rem}
	\label{discfracsemigroup}
	Lemma \ref{fracbound} also holds with a uniform constant $C$ (independent of $h$) when $A$ and $\mathcal{S}_i(t)$, $i=1,2$ are replaced respectively by their discrete versions $A_h$ and $\mathcal{S}_{ih}(t)$ defined in Section \ref{spa_convpaper4}.
\end{rem}
In order to ensure the existence and the uniqueness of mild solution for SPDE \eqref{pbfrac} and for the purpose of convergence analysis we make the following assumptions.
\begin{ass}[Initial Value]
	\label{init}
	We assume that $X_0:\Omega\rightarrow H$ is $\mathcal{F}_0/\mathcal{B}(\mathcal{H})$-measurable mapping and $X_0\in L^2(\Omega, D(A^{\frac{2H+\beta-1}{2}}))$ with $\beta\in (1-2H,1]$, $\frac 12<\alpha<1$.
\end{ass}
\begin{ass}[Non linearity term $F$]
	\label{nonlin}
	We assume the non-linear mapping $F: \mathcal{H}\rightarrow \mathcal{H}$, to be linear growth and Lipschitz continuous ie, for $\kappa\in[0,1]$ there exists positive constant $L > 0$ such that
	\begin{eqnarray}
		\|F(u)-F(v)\|^2\leq L\|u-v\|^2,\quad\|A^{\kappa}F(u)\|^2\leq L\left(1+\|A^{\kappa}u\|^2\right),\quad u,v\in \mathcal{H}.
	\end{eqnarray}
\end{ass}
\begin{ass}[Standard Noise term]
	\label{diff}
	We assume that the diffusion coefficient $G: \mathcal{H}\rightarrow L^0_2$ satisfies the global Lipschitz condition and the linear growth ie, for $\tau\in[0,1]$, there exists a positive constant $L>0$ such that 
	\begin{eqnarray}
		\|G(u)-G(v)\|_{L^0_2}^2\leq L\|u-v\|^2,\hspace{1cm}\|A^{\tau}G(u)\|_{L^0_2}^2\leq L\left(1+\|A^{\tau}u\|^2\right),\hspace{0.5cm}u,v\in \mathcal{H},
	\end{eqnarray}
\end{ass} 
\begin{ass}[Fractional Noise term]
	\label{fracdiff}
	The deterministic mapping $\Phi:\mathcal{H}\rightarrow \mathcal{H}$ satisfies
	\begin{eqnarray}
		\|A^{\frac{\beta-1}2}\Phi\|^2_{L_2^0}<\infty ,
	\end{eqnarray}
	where $\beta$ is defined as in \assref{init}.
\end{ass} 
We are no to present the following result of existence and uniqueness of mild solution of SPDE \eqref{pbfrac}. 
\begin{thm}
	\label{well_poss_pbfrac}
	Under the Assumptions \ref{init}-\ref{fracdiff}, if   $2L\left[\frac{T^{2\alpha-1} }{2\alpha-1}\left(\frac{\alpha\Gamma(2)}{\Gamma(1+\alpha)}\right)^2+1\right]<1$  the SPDEs \eqref{pbfrac} admits a unique mild solution $X(t)\in L^2(\Omega\times[0,T],\mathcal{H})$ asymptotic stable in mean square ie
	\begin{eqnarray}
		\label{boun}
		\mathbb{E}\left[\sup_{0\leq t\leq T}\left\|X(t)\right\|^2\right]<\infty,
	\end{eqnarray}
	where $L^2(\Omega\times[0,T],\mathcal{H})$ denotes the space of squared integrable $\mathcal{H}$-valued random variables, $c_0$ is the constant from the bilinear form associated to $A$, $\alpha$ is the order of the Caputo derivative and $L$ is the Lipschitz condition from Assumption \ref{diff} defined on $[0,T]$.
\end{thm}
\begin{proof}
	We define the operator $\xi: L^2(\Omega\times[0,T],\mathcal{H}) \rightarrow L^2(\Omega\times[0,T],\mathcal{H})$ by,
	\begin{eqnarray}
		\label{op}
		\xi x(t)&=&\mathcal{S}_1(t)x_0+\int_0^t (t-s)^{\alpha-1}\mathcal{S}_2(t-s)F(x(s))ds+\int_0^t \mathcal{S}_1(t-s)G(x(s))dW(s)\nonumber\\
		&+&\int_0^t \mathcal{S}_1(t-s)\Phi dB^H(s).
	\end{eqnarray}
	In order to obtain our result, we use the Banach fixed point to prove that the mapping $\xi$ has a unique fixed point in $L^2(\Omega\times[0,T],\mathcal{H})$. The proof will be splitted into two steps.
	\begin{enumerate}
		\item[\textbf{Step 1:}] First, we show that $\xi \in L^2(\Omega\times[0,T],\mathcal{H}) \subset L^2(\Omega\times[0,T],\mathcal{H})$.\\
		Let $x\in L^2(\Omega\times[0,T],\mathcal{H})$, using \eqref{op}, triangle inequality and the estimate $\left(\sum_{i=1}^n a_i\right)^2\leq n\sum_{i=1}^n a_i^2$, we have
		\begin{eqnarray}
			\label{xi}
			\mathbb{E}[\left\|\xi x(t)\right\|^2]&\leq& 4\mathbb{E}\left\|\mathcal{S}_1(t)x_0\right\|^2+4\mathbb{E}\left\|\int_0^t (t-s)^{\alpha-1}\mathcal{S}_2(t-s)F(x(s))ds\right\|^2\nonumber\\
			&+&4\mathbb{E}\left\|\int_0^t \mathcal{S}_1(t-s)G(x(s))dW(s)\right\|^2 +4 \mathbb{E}\left\|\int_0^t \mathcal{S}_1(t-s)\Phi dB^H(s) \right\|^2\nonumber\\
			&=:& 4\sum_{i=1}^4 I_i.
		\end{eqnarray}
		Using the fact that fractional semigroup $\mathcal{S}_1(t)$ is a contraction \eqref{contract1} yields
		\begin{eqnarray}
			\label{xi1}
			I_1:=\mathbb{E}\left\|\mathcal{S}_1(t)x_0\right\|^2 \leq \mathbb{E}\left\|x_0\right\|^2<\infty.
		\end{eqnarray}
		Using Cauchy-Schwarz inequality, the stability property of fractional semigroup $\mathcal{S}_2(t)$ \eqref{contract1} and Assumption \ref{nonlin} with $\kappa=0$ yields
		\begin{eqnarray}
			\label{xi2}
			I_2&:=&\mathbb{E}\left\|\int_0^t (t-s)^{\alpha-1}\mathcal{S}_2(t-s)F(x(s))ds\right\|^2\nonumber\\
			&\leq& L\left(\frac{\alpha\Gamma(2)}{\Gamma(1+\alpha)}\right)^2\left(\int_0^t(t-s)^{2\alpha-2}ds\right)\int_0^t \mathbb{E}(1+\left\|x(s)\right\|^2)ds\nonumber\\
			&\leq& \frac{c_0^2 \,t^{2\alpha}}{2\alpha-1}\left(\frac{\alpha\Gamma(2)}{\Gamma(1+\alpha)}\right)^2\left(1+\left\|x\right\|^2_{L^2(\Omega\times[0,T],\mathcal{H})}\right)<\infty.
		\end{eqnarray}
		Using also the contraction argument of the semigroup $\mathcal{S}_1(t)$ \eqref{contract1}, Ito isometry \eqref{browint} and Assumption \ref{diff} with $\tau=0$, hold
		\begin{eqnarray}
			\label{xi3}
			I_3&:=& \mathbb{E}\left\|\int_0^t \mathcal{S}_1(t-s)G(x(s))dW(s)\right\|^2\nonumber\\
			&=& \mathbb{E}\left[\int_0^t\left\|\mathcal{S}_1(t-s)G(x(s))\right\|^2_{L^0_2} ds\right] \nonumber\\
			&\leq& L\int_0^t \mathbb{E}(1+\left\|x(s)\right\|^2)ds\nonumber\\
			&\leq& L t\left(1+\left\|x\right\|^2_{L^2(\Omega\times[0,T],\mathcal{H})}\right)<\infty.
		\end{eqnarray}
		Applying \eqref{fracint}, inserting an appropriate power of $A$ and Lemma \ref{fracbound} with $\rho=\delta=\frac{1-\beta}2$, we have 
		\begin{eqnarray}
			\label{xi4}
			I_4&:=& \mathbb{E}\left\|\int_0^t \mathcal{S}_1(t-s)\Phi dB^H(s) \right\|^2\nonumber\\
			&\leq& C(H)\sum_{i=0}^{\infty}\left(\int_0^t\left\|\mathcal{S}_1(t-s)\Phi Q^{\frac 12}e_i\right\|^{\frac 1H}ds\right)^{2H}\nonumber\\
			&\leq& C(H) \sum_{i=0}^{\infty}\left(\int_0^t\left\|A^{\frac{1-\beta}2}\mathcal{S}_1(t-s)\right\|^{\frac 1H}_{L(\mathcal{H})}\left\| A^{\frac{\beta-1}2}\Phi Q^{\frac 12}e_i\right\|^{\frac 1H}ds\right)^{2H}\nonumber\\
			&\leq& C(H)\left(\int_0^t(t-s)^{\frac{-\alpha(1-\beta)}{2H}}ds\right)^{2H}\left( \sum_{i=0}^{\infty}\left\| A^{\frac{\beta-1}2}\Phi Q^{\frac 12}e_i\right\|^2\right)\nonumber\\
			&\leq& C(H)\left(1-\frac{\alpha(1-\beta)}{2H}\right)^{-2H}t^{2H-\alpha(1-\beta)}\left\|A^{\frac{\beta-1}2}\Phi\right\|^2_{L^0_2}<\infty.
		\end{eqnarray}
		
		Inserting \eqref{xi1}, \eqref{xi2}, \eqref{xi3} and \eqref{xi4} in \eqref{xi} implies that $\mathbb{E}\left\|\xi x(t)\right\|^2<\infty$ for all $t\in[0,T]$. Thus we conclude that $\xi x\in L^2(\Omega\times[0,T],\mathcal{H})$.
		
		\item[\textbf{Step 2:}] Next, we show that the mapping $\xi$ is contractive.\\
		To see this, let $x,y\in L^2(\Omega\times[0,T],\mathcal{H})$, then from \eqref{op} we get
		\begin{eqnarray}
			\label{step2}
			\mathbb{E}\left\|(\xi x)(t)-(\xi y)(t)\right\|^2&\leq&2\mathbb{E}\left\|\int_0^t (t-s)^{\alpha-1}\mathcal{S}_2(t-s)(F(x(s))-F(y(s)))ds\right\|^2\nonumber\\
			&+& 2\mathbb{E}\left\|\int_0^t \mathcal{S}_1(t-s)(G(x(s))-G(y(s)))dW(s)\right\|^2\nonumber\\
			&=:& 2J_1+2J_2.
		\end{eqnarray}
		Cauchy-Schwarz inequality, the stability property of fractional semigroup $\mathcal{S}_2(t)$ \eqref{contract1} and Assumption \ref{nonlin} yield
		\begin{eqnarray}
			\label{step21}
			J_1&:=&\mathbb{E}\left\|\int_0^t (t-s)^{\alpha-1}\mathcal{S}_2(t-s)(F(x(s))-F(y(s)))ds\right\|^2\nonumber\\
			&\leq&  L\left(\frac{\alpha\Gamma(2)}{\Gamma(1+\alpha)}\right)^2\left(\int_0^t(t-s)^{2\alpha-2}ds\right)\left(\int_0^t \mathbb{E}\left\|x(s)-y(s)\right\|^2ds\right)\nonumber\\
			&\leq& \frac{L \,t^{2\alpha-1}}{2\alpha-1}\left(\frac{\alpha\Gamma(2)}{\Gamma(1+\alpha)}\right)^2\int_0^t \mathbb{E}\left\|x(s)-y(s)\right\|^2ds.
		\end{eqnarray}
		Using the contraction argument of the semigroup, Ito isometry \eqref{browint} and Assumption \ref{diff}, we have 
		\begin{eqnarray}
			\label{step22}
			J_2&:=& \mathbb{E}\left\|\int_0^t \mathcal{S}_1(t-s)(G(x(s))-G(y(s)))dW(s)\right\|^2\nonumber\\
			&=& \mathbb{E}\left[\int_0^t \left\| \mathcal{S}_1(t-s)(G(x(s))-G(y(s)))\right\|^2_{L^0_2}ds\right]\nonumber\\
			&\leq& L \int_0^t \mathbb{E}\left\|x(s)-y(s)\right\|^2ds.
		\end{eqnarray}
		Hence putting \eqref{step21} and \eqref{step22} in \eqref{step2} holds
		\begin{eqnarray}
			\label{step2*}
			\mathbb{E}\left\|(\xi x)(t)-(\xi y)(t)\right\|^2\leq 2L\left[\frac{T^{2\alpha-1} }{2\alpha-1}\left(\frac{\alpha\Gamma(2)}{\Gamma(1+\alpha)}\right)^2+1\right]\int_0^t \mathbb{E}\left\|x(s)-y(s)\right\|^2ds.
		\end{eqnarray}
		\end{enumerate}
		The fact that $2L\left[\frac{T^{2\alpha-1} }{2\alpha-1}\left(\frac{\alpha\Gamma(2)}{\Gamma(1+\alpha)}\right)^2+1\right]<1$ and \eqref{step2*} imply that $\xi$ is a contraction mapping. 
	Thus, applying the Banach fixed point principle,  it follows that there exists a unique $x(t)\in L^2(\Omega;\mathcal{H})$ that solve the equation \eqref{pbfrac}, and $x(t)$ is asymptotic stable in mean square.
	The proof is thus completed. 
\end{proof}

In all that follows, $C$ denotes a positive constant that may change from line to line. In the Banach space $D(A^{\frac{\gamma}{2}})$, $\gamma\in\R$, we use the notation $\|A^{\frac{\gamma}{2}}\cdot\|=\|\cdot\|_{\gamma}$ and we now present the following regularity results.

\section{Regularity of the mild solution}
\label{regpaper4}
We discuss the space and time regularity of the mild solution $X(t)$ of \eqref{pbfrac} given by \eqref{solpbfrac} in this section. 
The following theorem presents the spatial   and  time regularity result.
\begin{thm}
	\label{regfrac}
	Under Assumptions \ref{init}-\ref{fracdiff}, the unique mild solution $X(t)$ given by \eqref{solpbfrac} satisfied the following space regularity
	\begin{eqnarray}
		\label{sparegfrac}
		\left\|A^{\frac{2H+\beta-1}2}X(t)\right\|_{L^2(\Omega;\mathcal{H})}\leq C\left(1+\left\|A^{\frac{2H+\beta-1}2}X_0\right\|_{L^2(\Omega;\mathcal{H})}\right),\hspace{1cm}t\in[0,T],
	\end{eqnarray}
	and for $0\leq t_1<t_2\leq T$, the following optimal time regularity hold
	\begin{eqnarray}
		\label{timeregfrac}
		\left\|X(t_2)-X(t_1)\right\|_{L^2(\Omega;\mathcal{H})}\leq C(t_2-t_1)^{\frac{\min(\alpha(2H+\beta-1),2-2\alpha)}2}\left(1+\left\|A^{\frac{2H+\beta-1}2}X_0\right\|_{L^2(\Omega;\mathcal{H})}\right) 
	\end{eqnarray}
	Moreover \eqref{sparegfrac} and \eqref{timeregfrac} hold when $A$ and $X$ are replaced by their semidiscrete version $A_h$ and $X^h$ defined in section \ref{spa_convpaper4}.
\end{thm}
\begin{proof}
	We begin by proving \eqref{sparegfrac}. Premultiplying \eqref{solpbfrac} by $A^{\frac{2H+\beta-1}2}$, taking the squared-norm and the estimate $\left(\sum_{i=1}^n a_i\right)^2\leq n\sum_{i=1}^n a_i^2$ yields
	\begin{eqnarray}
		\label{spa}
		\left\|A^{\frac{2H+\beta-1}2}X(t)\right\|^2_{L^2(\Omega;\mathcal{H})}&\leq& 4\left\|A^{\frac{2H+\beta-1}2}\mathcal{S}_1(t)X_0\right\|^2_{L^2(\Omega;\mathcal{H})}\nonumber\\
		&+& 4\left\|\int_0^t (t-s)^{\alpha-1}A^{\frac{2H+\beta-1}2}\mathcal{S}_2(t-s)F(X(s))ds\right\|^2_{L^2(\Omega;\mathcal{H})}\nonumber\\
		&+& 4\left\|\int_0^t A^{\frac{2H+\beta-1}2}\mathcal{S}_1(t-s)G(X(s))dW(s)\right\|^2_{L^2(\Omega;\mathcal{H})}\nonumber\\
		&+& 4\left\|\int_0^t A^{\frac{2H+\beta-1}2}\mathcal{S}_1(t-s)\Phi dB^H(s) \right\|^2_{L^2(\Omega;\mathcal{H})}\nonumber\\
		&=:& 4\sum_{i=1}^4 II_i.
	\end{eqnarray}
	We bound $II_i, i=1,2,3$ one by one. \\
	Firstly, using the fact that $\mathcal{S}_1(t)$ is a contraction \eqref{contract1} and \eqref{semigroup_prp*} with $\delta=\frac{2H+\beta-1}2$ yields
	\begin{eqnarray}
		\label{spa1}
		II_1:=\left\|A^{\frac{2H+\beta-1}2}\mathcal{S}_1(t)X_0\right\|^2_{L^2(\Omega;\mathcal{H})}\leq \left\|A^{\frac{2H+\beta-1}2}X_0\right\|^2_{L^2(\Omega;\mathcal{H})}.
	\end{eqnarray} 
	Secondly, using Cauchy-Schwarz inequality, the stability property of fractional semigroup $\mathcal{S}_2(t)$ \eqref{contract1}, \eqref{semigroup_prp*} and Assumption \ref{nonlin} with $\delta=\kappa=\frac{2H+\beta-1}2$ yields
	\begin{eqnarray}
		\label{spa2}
		II_2&:=&\mathbb{E}\left[\left\|\int_0^t (t-s)^{\alpha-1}A^{\frac{2H+\beta-1}2}\mathcal{S}_2(t-s)F(X(s))ds\right\|^2\right]\nonumber\\
		&\leq& L\left(\int_0^t(t-s)^{2\alpha-2}\left\|\mathcal{S}_2(t-s)\right\|^2_{L(\mathcal{H})}ds\right)\int_0^t \left(1+\mathbb{E}\left\|A^{\frac{2H+\beta-1}2}X(s)\right\|^2\right)ds\nonumber\\
		&\leq& \frac{L \,t^{2\alpha-1}}{2\alpha-1}\left(\frac{\alpha\Gamma(2)}{\Gamma(1+\alpha)}\right)^2\left(t+\int_0^t \mathbb{E}\left\|A^{\frac{2H+\beta-1}2}X(s)\right\|^2ds\right)\nonumber\\
		&\leq& C+C\int_0^t\left\|A^{\frac{2H+\beta-1}2}X(s)\right\|^2_{L^2(\Omega,\mathcal{H})}ds.
	\end{eqnarray}
	Thirdly, using the It\^o isometry \eqref{browint}, the contraction of $\mathcal{S}_1(t)$ \eqref{contract1}, \eqref{semigroup_prp*} and Assumption \ref{diff} with $\delta=\tau=\frac{2H+\beta-1}2$ holds
	\begin{eqnarray}
		\label{spa3}
		II_3&:=& \mathbb{E}\left[\left\|\int_0^t A^{\frac{2H+\beta-1}2}\mathcal{S}_1(t-s)G(X(s))dW(s)\right\|^2\right]\nonumber\\
		&=& \mathbb{E}\left[\int_0^t\left\| A^{\frac{2H+\beta-1}2}\mathcal{S}_1(t-s)G(X(s))\right\|^2_{L^0_2} ds\right]\nonumber\\
		&\leq& \mathbb{E}\left[\int_0^t\left\| A^{\frac{2H+\beta-1}2}G(X(s))\right\|^2_{L^0_2} ds\right]\nonumber\\
		&\leq& C\mathbb{E}\left[\int_0^t 1+\left\| A^{\frac{2H+\beta-1}2}X(s)\right\|^2 ds\right]\nonumber\\
		&\leq& C+C\int_0^t \left\| A^{\frac{2H+\beta-1}2}X(s)\right\|^2_{L^2(\Omega;\mathcal{H})} ds.
	\end{eqnarray}
	Fourthly, using \eqref{fracint}, Lemma \ref{fracbound} \eqref{semigroup_prp} with $\rho=H$, \eqref{semigroup_prp*} with $\delta=1-\beta$ and Assumption \ref{fracdiff}, we have 
	\begin{eqnarray}
		\label{spa4}
		II_4&:=& \mathbb{E}\left[4\left\|\int_0^t A^{\frac{2H+\beta-1}2}\mathcal{S}_1(t-s)\Phi dB^H(s) \right\|^2\right]\nonumber\\
		&\leq& C(H)\sum_{i=0}^{\infty}\left(\int_0^t \left\|A^{\frac{2H+\beta-1}2}\mathcal{S}_1(t-s)\Phi Q^{\frac 12}e_i\right\|^{\frac 1H}ds \right)^{2H}\nonumber\\
		&\leq& C(H)\sum_{i=0}^{\infty}\left(\int_0^t \left\|A^H\mathcal{S}_1(t-s)\right\|_{L(\mathcal{H})}^{\frac 1H}\left\|A^{\frac{\beta-1}2}\Phi Q^{\frac 12}e_i\right\|^{\frac 1H}ds \right)^{2H}\nonumber\\
		&\leq& C(H)\left(\int_0^t(t-s)^{-\alpha}ds \right)^{2H} \left(\sum_{i=0}^{\infty}\left\|A^{\frac{\beta-1}2}\Phi Q^{\frac 12}e_i\right\|^2\right)\nonumber\\
		&\leq& C(H)\frac{t^{2H(1-\alpha)}}{(1-\alpha)^{2H}} \left\|A^{\frac{\beta-1}2}\Phi\right\|^2_{L^0_2}\leq C.
	\end{eqnarray}
	Inserting \eqref{spa1}-\eqref{spa4} in \eqref{spa} hence yields
	\begin{eqnarray*}
		\left\| A^{\frac{2H+\beta-1}2}X(t)\right\|^2_{L^2(\Omega;\mathcal{H})}\leq C\left(1+\left\| A^{\frac{2H+\beta-1}2}X_0\right\|^2_{L^2(\Omega;\mathcal{H})}\right)+C\int_0^t\left\| A^{\frac{2H+\beta-1}2}X(s)\right\|^2_{L^2(\Omega;\mathcal{H})}ds.
	\end{eqnarray*}
	Applying continuous Gronwall's lemma proves \eqref{sparegfrac}.\\
	Now for the proof of \eqref{timeregfrac}, we rewrite the mild solution \eqref{solpbfrac} at times $t=t_2$ and $t=t_1$ and we subtract $X(t_2)$ by $X(t_1)$ as
	\begin{eqnarray}
		\label{tim_reg*}
		&&X(t_2)-X(t_1)\nonumber\\
		&=&\left(\mathcal{S}_1(t_2)-\mathcal{S}_1(t_1)\right)X_0+\int_{0}^{t_1}\left[(t_2-s)^{\alpha-1}\mathcal{S}_2(t_2-s)-(t_1-s)^{\alpha-1}\mathcal{S}_2(t_1-s)\right]F(X(s))ds\nonumber\\
		&+&\int_{0}^{t_1}\left[\mathcal{S}_1(t_2-s)-\mathcal{S}_1(t_1-s)\right]G(X(s))dW(s)+\int_{0}^{t_1}\left[\mathcal{S}_1(t_2-s)-\mathcal{S}_1(t_1-s)\right]\Phi\,dB^H(s)\nonumber\\
		&+&\int_{t_1}^{t_2}(t_2-s)^{\alpha-1}\mathcal{S}_2(t_2-s)F(X(s))\,ds+\int_{t_1}^{t_2}\mathcal{S}_1(t_2-s)G(X(s))\,dW(s)\nonumber\\
		&+&\int_{t_1}^{t_2}\mathcal{S}_1(t_2-s)\Phi\,dB^H(s).
	\end{eqnarray}
	Taking the $L^2$-norm in both sides and using triangle inequality yields
	\begin{eqnarray}
		\label{tim}
		\left\|X(t_2)-X(t_1)\right\|_{L^2(\Omega;\mathcal{H})}\leq \sum_{i=1}^7 III_i. 
	\end{eqnarray}
	Inserting an appropriate power of $A$, using Lemma \ref{fracbound} more precisely \eqref{semigroup_prp} and \eqref{semigroup_prp*} with $\eta=\delta=\frac{2H+\beta-1}2$, Assumption \ref{init} implies
	\begin{eqnarray}
		\label{tim1}
		III_1&:=& \left\|\left(\mathcal{S}_1(t_2)-\mathcal{S}_1(t_1)\right)X_0\right\|_{L^2(\Omega;\mathcal{H})}\nonumber\\
		&\leq& \left\|A^{-\frac{2H+\beta-1}2}\left(\mathcal{S}_1(t_2)-\mathcal{S}_1(t_1)\right)A^{\frac{2H+\beta-1}2}X_0\right\|_{L^2(\Omega;\mathcal{H})}\nonumber\\
		&\leq& C(t_2-t_1)^{\frac{\alpha(2H+\beta-1)}2}\left\|A^{\frac{2H+\beta-1}2}X_0\right\|_{L^2(\Omega;\mathcal{H})}.
	\end{eqnarray}
	The estimate of $III_2$ and $III_5$ and already obtained in \cite[(42), (45)]{Nouc} then we have
	\begin{eqnarray}
		\label{tim2}
		III_2\leq C(t_2-t_1)^{1-\alpha}.
	\end{eqnarray}
	and
	\begin{eqnarray}
		\label{tim5}
		III_5\leq C(t_2-t_1)^{\alpha}.
	\end{eqnarray}
	Using the It\^o isometry property \eqref{browint}, \eqref{semigroup_prp}, \eqref{semigroup_prp*}, Assumption \ref{diff} with $\eta=\delta=\tau=\frac{2H+\beta-1}2$ and \eqref{sparegfrac}, we obtain
	\begin{eqnarray}
		\label{tim3}
		III_3^2&=&\left\|\int_{0}^{t_1}\left[\mathcal{S}_1(t_2-s)-\mathcal{S}_1(t_1-s)\right]G(X(s))dW(s)\right\|^2_{L^2(\Omega;H)}\nonumber\\
		&=&	\mathbb{E}\left[\int_{0}^{t_1}\left\|A^{-\frac{2H+\beta-1}2}\left[\mathcal{S}_1(t_2-s)-\mathcal{S}_1(t_1-s)\right]A^{\frac{2H+\beta-1}2}G(X(s))\right\|^2ds\right]\nonumber\\
		&\leq&	C\int_{0}^{t_1}(t_2-t_1)^{\alpha(2H+\beta-1)}\mathbb{E}\left[\left\|A^{\frac{2H+\beta-1}2}G(X(s))\right\|^2_{L^0_2}\right]ds\nonumber\\
		&\leq&	C\int_{0}^{t_1}(t_2-t_1)^{\alpha(2H+\beta-1)}\left(1+\mathbb{E}\left[\|A^{\frac{2H+\beta-1}2}X(s)\|^2\right]\right)ds\nonumber\\
		&\leq& C(t_2-t_1)^{\alpha(2H+\beta-1)}t_1\left(1+\left\|A^{\frac{2H+\beta-1}2}X_0\right\|^2_{L^2(\Omega;\mathcal{H})}\right)\nonumber\\
		&\leq& C(t_2-t_1)^{\alpha(2H+\beta-1)}\left(1+\left\|A^{\frac{2H+\beta-1}2}X_0\right\|^2_{L^2(\Omega;\mathcal{H})}\right).
	\end{eqnarray}
	Let estimate now $III_4^2$. Using \eqref{fracint}, inserting an appropriate power of $A$, \eqref{semigroup_prp1*} with $\kappa=\frac{1-\beta}2$ and $\varpi=H$, Assumption \ref{fracdiff} yields
	\begin{eqnarray}
		\label{tim4}
		III_4^2&:=&\left\|\int_{0}^{t_1}\left[\mathcal{S}_1(t_2-s)-\mathcal{S}_1(t_1-s)\right]\Phi\,dB^H(s)\right\|^2_{L^2(\Omega;\mathcal{H})}\nonumber\\
		&=& \mathbb{E}\left\|\int_0^{t_1}\left[\mathcal{S}_1(t_2-s)-\mathcal{S}_1(t_1-s)\right]\Phi\,dB^H(s)\right\|^2\nonumber\\
		&\leq& C(H)\sum_{i=0}^{\infty}\left(\int_0^{t_1}\left\|A^{\frac{1-\beta}2}\left[\mathcal{S}_1(t_2-s)-\mathcal{S}_1(t_1-s)\right]A^{\frac{\beta-1}2}\Phi Q^{\frac 12}e_i\right\|^{\frac 1H}\,ds\right)^{2H}\nonumber\\
		&\leq& C\sum_{i=0}^{\infty}\left(\int_0^{t_1}\left[(t_2-t_1)^{\alpha\left(H-\frac{1-\beta}2\right)}(t_1-s)^{-\alpha H}\right]^{\frac 1H}\left\|A^{\frac{\beta-1}2}\Phi Q^{\frac 12}e_i\right\|^{\frac 1H}\,ds\right)^{2H}\nonumber\\
		&\leq& C\sum_{i=0}^{\infty}\left(\int_0^{t_1}\left[(t_2-t_1)^{\alpha\left(H-\frac{1-\beta}2\right)}(t_1-s)^{-\alpha H}\right]^{\frac 1H}\left\|A^{\frac{\beta-1}2}\Phi Q^{\frac 12}e_i\right\|^{\frac 1H}\,ds\right)^{2H}\nonumber\\
		&\leq& C(t_2-t_1)^{\alpha(2H+\beta-1)}\left(\int_0^{t_1}(t_1-s)^{-\alpha}ds\right)^{2H}\left\|A^{\frac{\beta-1}2}\Phi\right\|^{2}_{L^0_2}\nonumber\\
		&\leq&C(t_2-t_1)^{\alpha(2H+\beta-1)}t_1^{2H(1-\alpha)}\nonumber\\
		&\leq&C(t_2-t_1)^{\alpha(2H+\beta-1)}.
	\end{eqnarray}
	Now for the sixth term, we use It\^o isometry property \eqref{browint}, the contraction of $\mathcal{S}_1(t)$, Assumption \ref{diff} with $\tau=0$ and Theorem \ref{well_poss_pbfrac} to obtain
	\begin{eqnarray}
		\label{tim6}
		III_6^2&:=&\left\|\int_{t_1}^{t_2}\mathcal{S}_1(t_2-s)G(X(s))\,dW(s)\right\|^2_{L^2(\Omega;\mathcal{H})}\nonumber\\
		&=& \mathbb{E}\left[\int_{t_1}^{t_2}\left\|\mathcal{S}_1(t_2-s)G(X(s))\right\|^2_{L^0_2}\,ds\right]\nonumber\\
		&\leq& \int_{t_1}^{t_2}\left\|\mathcal{S}_1(t_2-s)\right\|^2_{L(\mathcal{H})} \mathbb{E}\left\|G(X(s))\right\|^2_{L^0_2}\,ds\nonumber\\ 
		&\leq& \left(\int_{t_1}^{t_2}ds\right)\left(1+\mathbb{E}\left[\sup_{0\leq s\leq T}\left\|X(s)\right\|^2\right]\right)\nonumber\\
		&\leq& C(t_2-t_1).
	\end{eqnarray}
	Finally, using \eqref{fracint}, inserting an appropriate power of $A$, Lemma \ref{fracbound} with $\rho=\delta=\frac{1-\beta}2$ and Assumption \ref{fracdiff} yields
	\begin{eqnarray}
		\label{tim7}
		III_7^2&:=&\left\|\int_{t_1}^{t_2}\mathcal{S}_1(t_2-s)\Phi\,dB^H(s)\right\|^2_{L^2(\Omega;\mathcal{H})}\nonumber\\
		&=& \mathbb{E}\left\|\int_{t_1}^{t_2}\mathcal{S}_1(t_2-s)\Phi\,dB^H(s)\right\|^2\nonumber\\
		&\leq& C(H)\sum_{i=0}^{\infty}\left(\int_{t_1}^{t_2}\left\|\mathcal{S}_1(t_2-s)\Phi Q^{\frac 12}e_i\right\|^{\frac 1H}\,ds\right)^{2H}\nonumber\\
		&\leq& C(H)\sum_{i=0}^{\infty}\left(\int_{t_1}^{t_2}\left\|A^{\frac{1-\beta}2}\mathcal{S}_1(t_2-s)\right\|^{\frac 1H}_{L(\mathcal{H})}\left\|A^{\frac{\beta-1}2}\Phi Q^{\frac 12}e_i\right\|^{\frac 1H}\,ds\right)^{2H}\nonumber\\
		&\leq& C(H)\left(\int_{t_1}^{t_2}(t_2-s)^{-\frac{\alpha(1-\beta)}{2H}}\,ds\right)^{2H} \left\|A^{\frac{\beta-1}2}\Phi\right\|^2_{L^0_2}\nonumber\\
		&\leq& C(t_2-t_1)^{2H+\alpha(\beta-1)}\leq C(t_2-t_1)^{\alpha(2H+\beta-1)}.
	\end{eqnarray}
	Substituting \eqref{tim1}-\eqref{tim7} in \eqref{tim} yields
	\begin{eqnarray*}
		\left\|X(t_2)-X(t_1)\right\|_{L^2(\Omega;\mathcal{H})}\leq C(t_2-t_1)^{\frac{\min(\alpha(2H+\beta-1),2-2\alpha)}2}\left(1+\left\|A^{\frac{2H+\beta-1}2}X_0\right\|_{L^2(\Omega;\mathcal{H})}\right) 
	\end{eqnarray*}
	The prof of Theorem \ref{regfrac} is thus completed.
\end{proof}
\section{Space approximation and error estimates}
\label{spa_convpaper4}
We consider the discretization of the spatial domain by a finite element triangulation with maximal length $h$ satisfying the usual regularity assumptions. Let $V_h\subset \mathcal{H}$ denote the space of continuous functions that are piecewise linear over triangulation $J_h$. To discretise in space, we introduce $P_h$ from $L^2(\Omega)$ to $V_h$ define for $u\in L^2(\Omega)$ by
\begin{eqnarray}
	\label{pro}
	(P_h u,\xi)=(u,\xi),\hspace{2cm}\forall \xi\in V_h.
\end{eqnarray}
The discrete operator $A_h: V_h\rightarrow V_h$ is defined by
\begin{eqnarray}
	\label{dis}
	(A_h \rho,\xi)=-a(\rho,\xi),\hspace{2cm}\forall \rho, \xi\in V_h,
\end{eqnarray}
where $a$ is the corresponding bilinear form of $A$. Like the operator $A$, the discrete operator $A_h$ is also the generator of a contraction semigroup $S_h(t):=e^{-tA_h}$. The semidiscrete space version of problem \eqref{pbfrac2} is to find $X^h(t)=X^h(\cdot,t)$ such that for $t\in(0,T]$
\begin{eqnarray}
	\label{modeldisfrac4}
	\left\{
	\begin{array}{ll}
		^C\partial_t^{\alpha}X^h(t)+A^h X^h(t)=P_hF(X^h(t))+I_t^{1-\alpha}\left[P_hG(X^h(t))\frac{dW(t)}{dt}+P_h\Phi\frac{dB^H(t)}{dt}\right], \\
		X^h(0)=P_hX_0.
	\end{array}
	\right.
\end{eqnarray}
Note that $A_h$, $P_hG$ and $P_h\Phi$ satisfy the same assumptions as $A$, $G$ and $\Phi$ respectively. The mild solution of \eqref{modeldisfrac4} can be represented as follows
\begin{eqnarray}
	\label{dissolfrac}
	X^h(t)&=&\mathcal{S}_{1h}(t)X_0+\int_0^t (t-s)^{\alpha-1}\mathcal{S}_{2h}(t-s)P_hF(X^h(s))ds+\int_{0}^{t}\mathcal{S}_{1h}(t-s)P_hG(X^h(s))dW(s)\nonumber\\
	&+&\int_{0}^{t}\mathcal{S}_{1h}(t-s)P_h\Phi dB^H(s).  
\end{eqnarray}
Where $\mathcal{S}_{1h}$ and $\mathcal{S}_{2h}$ are the semidiscrete version of $\mathcal{S}_1$ and $\mathcal{S}_2$ respectively defined by \eqref{S1} and \eqref{S2}.

Let us define the error operators 
\begin{eqnarray*}
	\mathcal{T}_{1h}(t):=\mathcal{S}_1(t)-\mathcal{S}_{1h}(t)P_h,\hspace{1cm}\mathcal{T}_{2h}(t):=\mathcal{S}_2(t)-\mathcal{S}_{2h}(t)P_h.
\end{eqnarray*}
Then we have the following Lemma.
\begin{lem}(\cite[Lemma 3]{Nouc})
	\label{disop}
	\begin{enumerate}
		\item [(i)] Let $r\in[0,2]$, $\rho\leq r$, $t\in(0,T]$, $v\in D(A^{\rho})$. Then there exists a positive constant C such that
		\begin{eqnarray}
			\label{fradisop1}
			\|\mathcal{T}_{1h}(t)v\|\leq C h^rt^{-\alpha(r-\rho)/2}\left\|v\right\|_{\rho},\hspace{1cm}\|\mathcal{T}_{2h}(t)v\|\leq C h^rt^{-\alpha(r-\rho)/2}\left\|v\right\|_{\rho}.
		\end{eqnarray}
		\item [(ii)] Let $0\leq\gamma\leq 1$, then there exists a constant $C$ such that 
		\begin{eqnarray}
			\label{fradisop2}
			\left\|\int_0^ts^{\alpha-1}\mathcal{T}_{2h}(s)v ds\right\|\leq C h^{2-\gamma} \left\|v\right\|_{-\gamma},\hspace{1cm}v\in D(A^{-\gamma}),\,\,t>0.
		\end{eqnarray}
	\end{enumerate}
\end{lem}

The following lemma provides an estimate in mean square sense for the error between the solution of SPDE \eqref{pbfrac2} and the spatially semidiscrete approximation \eqref{dissolfrac}.
\begin{lem}[Space error]
	\label{spaerrfrac}
	Let $X$ and $X^h$ be the mild solution of \eqref{pbfrac2} and \eqref{modeldisfrac4}, respectively. Let Assumptions \ref{init} - \ref{fracdiff} be fulfilled then there exits a constant $C$ independent of $h$, such that 
	\begin{eqnarray}
		\label{spacon}
		\left\|X(t)-X^h(t)\right\|_{L^2(\Omega;\mathcal{H})}\leq C h^{2H+\beta-1}, \hspace{1cm} 0\leq t\leq T. 
	\end{eqnarray}
\end{lem}
\begin{proof}
	Define $e(t):=X(t)-X^h(t)$. By \eqref{solpbfrac} and \eqref{dissolfrac}, taking the norm, using triangle inequality and the estimate $\left(\sum_{i=1}^na_i\right)^2\leq n\sum_{i=1}^na_i^2 $ we deduce that
	\begin{eqnarray}
		\label{spaer}
		&&\left\|e(t)\right\|^2_{L^2(\Omega;\mathcal{H})}\nonumber\\
		&\leq&4\left\|\mathcal{S}_1(t)X_0-\mathcal{S}_{1h}(t)P_hX_0\right\|^2_{L^2(\Omega;\mathcal{H})}\nonumber\\
		&+&4\left\|\int_0^t\left[(t-s)^{\alpha-1}\mathcal{S}_2(t-s)X(s)-(t-s)^{\alpha-1}\mathcal{S}_{2h}(t-s)P_hF(X^h(s))\right]\,ds\right\|^2_{L^2(\Omega;\mathcal{H})} \nonumber\\
		&+& 4 \left\|\int_0^t\left[\mathcal{S}_1(t-s)G(X(s))-\mathcal{S}_{1h}(t-s)P_h G(X^h(s))\right]\,dW(s)\right\|^2_{L^2(\Omega;\mathcal{H})} \nonumber\\
		&+& 4 \left\|\int_0^t\left[\mathcal{S}_1(t-s)-\mathcal{S}_{1h}(t-s)P_h\right]\Phi\,dB^H(s)\right\|^2_{L^2(\Omega;\mathcal{H})} \nonumber\\
		&=:& 4\sum_{i=1}^4 IV_i.
	\end{eqnarray} 
	We will analyse the above terms $IV_i, i=1,2,3,4$ one by one.\\
	For the first term $IV_1$, using Lemma \ref{disop} with $r=\gamma=2H+\beta-1$ and Assumption \ref{init} yields
	\begin{eqnarray}
		\label{spaer1}
		IV_1&:=&\left\|\mathcal{S}_1(t)X_0-\mathcal{S}_{1h}(t)P_hX_0\right\|^2_{L^2(\Omega;\mathcal{H})}\nonumber\\
		&=&\left\|\left[\mathcal{S}_1(t)-\mathcal{S}_{1h}(t)P_h\right]X_0\right\|^2_{L^2(\Omega;\mathcal{H})}\nonumber\\
		&=& \left\|\mathcal{T}_{1h}(t)X_0\right\|^2_{L^2(\Omega;\mathcal{H})}\nonumber\\
		&\leq& C\,h^{2(2H+\beta-1)}\left\|A^{\frac{2H+\beta-1}2}X_0\right\|^2_{L^2(\Omega;\mathcal{H})}\nonumber\\
		&\leq& C\,h^{2(2H+\beta-1)}.
	\end{eqnarray}
	For the second term $IV_2$, by adding and subtracting a term, applying triangle inequality and the estimate $(a+b)^2\leq 2a^2+2b^2$, we split it in two terms as follows 
	\begin{eqnarray}
		\label{spaer2*}
		IV_2&:=& \left\|\int_0^t\left[(t-s)^{\alpha-1}\mathcal{S}_2(t-s)F(X(s))-(t-s)^{\alpha-1}\mathcal{S}_{2h}(t-s)P_hF(X^h(s))\right]\,ds\right\|^2_{L^2(\Omega;\mathcal{H})} \nonumber\\
		&\leq& 2 \left\|\int_0^t(t-s)^{\alpha-1}\left[\mathcal{S}_2(t-s)-\mathcal{S}_{2h}(t-s)P_h\right]F(X(s))\,ds\right\|^2_{L^2(\Omega;\mathcal{H})} \nonumber\\
		&+& 2\left\|\int_0^t(t-s)^{\alpha-1}\mathcal{S}_{2h}(t-s)P_h( F(X(s))-F(X^h(s)))\,ds\right\|^2_{L^2(\Omega;\mathcal{H})} \nonumber\\
		&=:& 2IV_{21}+2IV_{22}
	\end{eqnarray}
	Firstly, adding and subtracting a term, applying Cauchy-Schwartz inequality, Lemma \ref{disop} (i) with $r=2H+\beta-1$, $\rho=0$, Lemma \eqref{regfrac} more precisely \eqref{timeregfrac} for the first term and Lemma \ref{disop} (ii) with $\gamma=0$, Lemma \ref{well_poss_pbfrac} and Assumption \ref{nonlin} with $\kappa=0$ yields
	\begin{eqnarray}
		\label{spaer21}
		IV_{21}&:=& \left\|\int_0^t(t-s)^{\alpha-1}\left[\mathcal{S}_2(t-s)-\mathcal{S}_{2h}(t-s)P_h\right]F(X(s))\,ds\right\|^2_{L^2(\Omega;\mathcal{H})} \nonumber\\
		&\leq& 2\left\|\int_0^t(t-s)^{\alpha-1}\mathcal{T}_{2h}(t-s)(F(X(s))-F(X(t)))\,ds\right\|^2_{L^2(\Omega;\mathcal{H})}\nonumber\\
		&+&2\left\|\int_0^t(t-s)^{\alpha-1}\mathcal{T}_{2h}(t-s)F(X(t))\,ds\right\|^2_{L^2(\Omega;\mathcal{H})} \nonumber\\
		&\leq& 2\left(\int_0^t(t-s)^{2\alpha-2}\,ds\right)\left(\int_0^t\mathbb{E}\left\|\mathcal{T}_{2h}(t-s)(F(X(s))-F(X(t)))\right\|^2\right)\,ds\nonumber\\
		&+&2\left\|\int_0^ts^{\alpha-1}\mathcal{T}_{2h}(s)F(X(t))\,ds\right\|^2_{L^2(\Omega;\mathcal{H})}\nonumber\\
		&\leq& C\,h^{2(2H+\beta-1)}\int_0^t(t-s)^{-\alpha(2H+\beta-1)}\left\|F(X(s))-F(X(t))\right\|^2_{L^2(\Omega;\mathcal{H})}\,ds\nonumber\\
		&+&C\,h^4\left\|F(X(t))\right\|^2_{L^2(\Omega;\mathcal{H})}\nonumber\\
		&\leq& C\,h^{2(2H+\beta-1)}\int_0^t(t-s)^{-\alpha(2H+\beta-1)}\left\|X(t)-X(s)\right\|^2_{L^2(\Omega;\mathcal{H})}\,ds\nonumber\\
		&+&C\,h^4\left(1+\left\|X(t)\right\|^2_{L^2(\Omega;\mathcal{H})}\right)\nonumber\\
		&\leq& C\,h^{2(2H+\beta-1)}\int_0^t(t-s)^{\min(0;1-\alpha(2H+\beta-1))}\,ds+C\,h^4\mathbb{E}\left[\sup_{0\leq t\leq T}\left\|X(t)\right\|^2\right]\nonumber\\
		&\leq& C\,h^{2(2H+\beta-1)}.
	\end{eqnarray}
	Secondly, applying the Cauchy-Schwartz inequality, boundedness of $P_h$, $\mathcal{S}_{2h}(t)$ and Assumption \ref{nonlin}, it holds
	\begin{eqnarray}
		\label{spaer22}
		IV_{22}&:=& \left\|\int_0^t(t-s)^{\alpha-1}\mathcal{S}_{2h}(t-s)P_h \left(F(X(s))-F(X^h(s))\right)\,ds\right\|^2_{L^2(\Omega;\mathcal{H})} \nonumber\\
		&\leq& \left(\int_0^t(t-s)^{2\alpha-2}\,ds\right)\left(\int_0^t\mathbb{E}\left\|\mathcal{S}_{2h}(t-s)P_h \left(F(X(s))-F(X^h(s))\right)\right\|^2\,ds\right) \nonumber\\
		&\leq& C \int_0^t\|e(s)\|^2_{L^2(\Omega;\mathcal{H})}\,ds.
	\end{eqnarray}
	Putting \eqref{spaer21} and \eqref{spaer22} in \eqref{spaer2*}, we obtain
	\begin{eqnarray}
		\label{spaer2}
		IV_2\leq C\,h^{2(2H+\beta-1)}+C\int_0^t\left\|e(s)\right\|^2_{L^2(\Omega;\mathcal{H})}.
	\end{eqnarray}
	For the third term $IV_3$, by adding and subtracting a term, using the triangle inequality and the estimate $(a+b)^2\leq 2a^2+2b^2$, we have
	\begin{eqnarray*}
		IV_3&=& \left\|\int_0^t\left[\mathcal{S}_1(t-s)G(X(s))-\mathcal{S}_{1h}(t-s)P_h G(X^h(s))\right]\,dW(s)\right\|^2_{L^2(\Omega;\mathcal{H})}\nonumber\\
		&\leq& 2\left\|\int_0^t\left[\mathcal{S}_1(t-s)-\mathcal{S}_{1h}(t-s)P_h\right]G(X(s))\,dW(s)\right\|^2_{L^2(\Omega;\mathcal{H})}\nonumber\\
		&+& 2 \left\|\int_0^t\mathcal{S}_{1h}(t-s)P_h (G(X(s))-G(X^h(s)))\,dW(s)\right\|^2_{L^2(\Omega;\mathcal{H})}.
	\end{eqnarray*}
	Applying It\^o isometry \eqref{browint}, Lemma \ref{disop} (i) with $r=\gamma=2H+\beta-1$ for the first term, Assumption \ref{diff} with $\tau=2H+\beta-1$, Theorem \ref{regfrac}, contraction argument of $\mathcal{S}_{1,h}(t)$ \eqref{contract1} and boundedness of $P_h$ it holds that
	\begin{eqnarray}
		\label{spaer3}
		IV_3&\leq& 2\int_0^t\mathbb{E}\left\|\mathcal{T}_{1h}G(X(s))\right\|^2_{L^0_2}\,ds+ \int_0^t\mathbb{E}\left\|\mathcal{S}_{1h}(t-s)P_h (G(X(s))-G(X^h(s)))\right\|^2_{L^0_2}\,ds\nonumber\\
		&\leq& C\,h^{2(2H+\beta-1)}\int_0^t\mathbb{E}\left\|A^{\frac{2H+\beta-1}2}G(X(s))\right\|^2_{L^0_2}\,ds\nonumber\\
		&+& C\int_0^t\mathbb{E}\left\| G(X(s))-G(X^h(s))\right\|^2_{L^0_2}\,ds\nonumber\\
		&\leq& C\,h^{2(2H+\beta-1)}\int_0^t\left(1+\left\|A^{\frac{2H+\beta-1}2}X(s)\right\|^2_{L^2(\Omega;\mathcal{H})}\right)\,ds+ C\int_0^t\left\| X(s)-X^h(s)\right\|^2_{L^2(\Omega;\mathcal{H})}\,ds\nonumber\\
		&\leq& C\,h^{2(2H+\beta-1)}\left(1+\left\|A^{\frac{2H+\beta-1}2}X_0\right\|^2_{L^2(\Omega;\mathcal{H})}\right)
		+ C\int_0^t\left\|e(s)\right\|^2_{L^2(\Omega;\mathcal{H})}\,ds\nonumber\\
		&\leq& C\,h^{2(2H+\beta-1)}+ C\int_0^t\left\|e(s)\right\|^2_{L^2(\Omega;\mathcal{H})}\,ds.
	\end{eqnarray}
	For the estimation of $IV_4$, \eqref{fracint}, Lemma \ref{disop} with $r=2H+\beta-1$ and $\gamma=\beta-1$, Assumption \ref{fracdiff} yields
	\begin{eqnarray}
		\label{spaer4}
		IV_4&=& \left\|\int_0^t\left[\mathcal{S}_1(t-s)-\mathcal{S}_{1h}(t-s)P_h\right]\Phi\,dB^H(s)\right\|^2_{L^2(\Omega;\mathcal{H})} \nonumber\\
		&\leq& C(H) \sum_{i=0}^{\infty}\left(\int_0^t\left\|\mathcal{T}_{1h}(t-s)\Phi Q^{\frac 12}e_i\right\|^{\frac 1H}\,ds\right)^{2H}\nonumber\\
		&\leq& C(H)\sum_{i=0}^{\infty}\left(\int_0^t C\,h^{\frac{2H+\beta-1}H}(t-s)^{-\alpha}\left\|A^{\frac{\beta-1}2}\Phi Q^{\frac 12}e_i\right\|^{\frac 1H}\,ds\right)^{2H}\nonumber\\
		&\leq& C\,h^{2(2H+\beta-1)}\left(\int_0^t (t-s)^{-\alpha}\,ds\right)^{2H}\left(\sum_{i=0}^{\infty}\left\|A^{\frac{\beta-1}2}\Phi Q^{\frac 12}e_i\right\|^2\right)\nonumber\\
		&\leq& C\,h^{2(2H+\beta-1)}t^{2H(1-\alpha)}\left\|A^{\frac{\beta-1}2}\Phi\right\|^2_{L^0_2}\nonumber\\
		&\leq& C\,h^{2(2H+\beta-1)}.
	\end{eqnarray}
	Putting \eqref{spaer1}, \eqref{spaer2}, \eqref{spaer3}, \eqref{spaer4} in \eqref{spaer} and applying Gronwall inequality ends the proof.  
\end{proof}

\section{Fully discrete Euler scheme and its error estimate}
\label{schemesfrac}
In this section, we consider a fully discrete approximation  scheme of SPDE \eqref{pbfrac2}, more precisely  a fractional exponential scheme. 
\subsection {Fractional exponential scheme}
Let $\Delta t$ the time step size and $M\in \N$ such that $T=M\Delta t$, hence for all $m\in\{0,1,...,M\}$, $X^h_m$ or $Y^h_m$ denotes the numerical approximation of $X^h(t_m)$ with $t_m=m\Delta t$. 
Since  we are  not longer dealing  with a semi group, the fractional exponential integrator scheme to \eqref{modeldisfrac4} extending the standard exponential  in \cite{Nou})  is giving by.
\begin{eqnarray}
	\label{numexpsol}
	X^h_{m}&=&\mathcal{S}_{1h}(t_m)P_hX_0+\Delta t\sum_{j=0}^{m-1}(t_m-t_j)^{\alpha-1}\mathcal{S}_{2h}(t_m-t_j)P_hF(X^h_j)\nonumber\\
	&+&\sum_{j=0}^{m-1}\mathcal{S}_{1h}(t_m-t_j)P_hG(X^h_j)\,\Delta W_j+\sum_{j=0}^{m-1}\mathcal{S}_{1h}(t_m-t_j)P_h\Phi\,\Delta B^H_j\\
	&=&E_{\alpha,1}(-t_m^{\alpha}A_h)P_hX_0+\Delta t\sum_{j=0}^{m-1}(t_m-t_j)^{\alpha-1}E_{\alpha,\alpha}(-(t_m-t_j)^{\alpha}A_h)P_hF(X^h_j)\nonumber\\
	&+&\sum_{j=0}^{m-1}E_{\alpha,1}(-(t_m-t_j)^{\alpha}A_h)P_hG(X^h_j)\,\Delta W_j+\sum_{j=0}^{m-1}E_{\alpha,1}(-(t_m-t_j)^{\alpha}A_h)P_h\Phi\,\Delta B^H_j,\nonumber
\end{eqnarray} 
where 
\begin{eqnarray*}
	\Delta W_j:= W(t_{j+1})-W(t_j)=\sum_{i=0}^{\infty}\sqrt{q_i}\left(\beta_i(t_{j+1})-\beta_i(t_j)\right)e_i
\end{eqnarray*}
and
\begin{eqnarray*}
	\Delta B^H_j:= B^H(t_{j+1})-B^H(t_j)=\sum_{i=0}^{\infty}\sqrt{q_i}\left(\beta_i^H(t_{j+1})-\beta_i^H(t_j)\right)e_i.
\end{eqnarray*}
As we can observe,   the analysis  will be more  complicated and different to that of  a standard  exponential integrator scheme \cite{Nou}(where $\alpha=1$)  since the fractional derivative is not local and therefore  numerical solution at $t_m$ depends to all previous numerical solutions up to $t_m$. This is   in contrast to the  standard  exponential  numerical scheme  where the  numerical solution  at  $t_m$ depends only of that  at  $t_{m-1}$. 

Our main result, which is indeed our full convergence result is given in the following theorem.
\begin{thm}
	\label{main}
	Let $X(t_m)$ be the mild solution of \eqref{pbfrac2} at time $t_m=m\Delta t$ given by \eqref{solpbfrac}. Let $X^h_m$ be the numerical approximation through \eqref{numexpsol}. Under Assumptions \ref{init}-\ref{fracdiff}, the following  estimation holds
	\begin{eqnarray}
		\left\|X(t_m)-X^h_m\right\|_{L^2(\Omega;\mathcal{H})}\leq C\left(h^{2H+\beta-1}+\Delta t^{\frac{\min(\alpha(2H+\beta-1),2-2\alpha)}2}\right).
	\end{eqnarray}
\end{thm}
Before giving the proof, let us present a preparatory result.
\begin{lem}({\cite[(89)]{Mukb} and \cite[(70)]{Tam}})
	\label{prepaexp2}
	Let $-1\leq \sigma\leq 1$, hence the following estimate holds
	\begin{eqnarray}
		\label{prepaexp21}
		\left\|A_h^{\sigma}P_h u\right\|\leq C\left\|A^{\sigma} u\right\|,   
	\end{eqnarray}
	where $C$ is a positive constant independent of $h$.
\end{lem} 
\subsection {Proof of Theorem \ref{main}}
We are now ready to prove the first result of our main theorem. In fact using the standard technique in the error analysis, we split the fully discrete error in two terms as
\begin{eqnarray}
	\label{err}
	\left\|X(t_m)-X^h_m\right\|_{L^2(\Omega;\mathcal{H})}&\leq& \left\|X(t_m)-X^h(t_m)\right\|_{L^2(\Omega;\mathcal{H})}+\left\|X^h(t_m)-X^h_m\right\|_{L^2(\Omega;\mathcal{H})}\nonumber\\
	&=:&  err_0+err_1.
\end{eqnarray}
Note that the space error $err_0$ is estimate by Lemma \ref{spaerrfrac}. It remains to estimate the time error $err_1$.
We recall that given the mild solution at time $t_m=m\Delta t$ of the semidiscrete problem \eqref{modeldisfrac4} is given by
\begin{eqnarray}
	\label{dismild}
	X^h(t_m)&=&\mathcal{S}_{1h}(t_m)X^h_0+\int_{0}^{t_m} (t_m-s)^{\alpha-1}\mathcal{S}_{2h}(t_m-s)P_hF(X^h(s))ds\nonumber\\
	&+&\int_{0}^{t_m}\mathcal{S}_{1h}(t_m-s)P_hG(X^h(s))dW(s) \nonumber\\
	&+&\int_{0}^{t_m}\mathcal{S}_{1h}(t_m-s)P_h\Phi dB^H(s).  
\end{eqnarray}
Decomposing \eqref{dismild}, we have
\begin{eqnarray}
	\label{itdismild}
	X^h(t_m)&=&\mathcal{S}_{1h}(t_m)X^h_0+\sum_{j=0}^{m-1}\int_{t_j}^{t_{j+1}} (t_m-s)^{\alpha-1}\mathcal{S}_{2h}(t_m-s)P_hF(X^h(s))ds\nonumber\\
	&+&\sum_{j=0}^{m-1}\int_{t_j}^{t_{j+1}}\mathcal{S}_{1h}(t_m-s)P_hG(X^h(s))dW(s)\nonumber\\
	&+&\sum_{j=0}^{m-1}\int_{t_j}^{t_{j+1}}\mathcal{S}_{1h}(t_m-s)P_h\Phi dB^H(s),
\end{eqnarray}
and thanks to \eqref{numexpsol}, we rewrite the numerical solution $X^h_m$ in the integral form as
\begin{eqnarray}
	\label{itnumexpsol}
	X^h_m&=&\mathcal{S}_{1h}(t_m)P_hX_0+\Delta t\sum_{j=0}^{m-1}(t_m-t_j)^{\alpha-1}\mathcal{S}_{2h}(t_m-t_j)P_hF(X^h_j)\nonumber\\
	&+&\sum_{j=0}^{m-1}\mathcal{S}_{1h}(t_m-t_j)P_hG(X^h_j)\,\Delta W_j+\sum_{j=0}^{m-1}\mathcal{S}_{1h}(t_m-t_j)P_h\Phi\,\Delta B^H_j\nonumber\\
	&=&\mathcal{S}_{1h}(t_m)X^h_0+\sum_{j=0}^{m-1}\int_{t_j}^{t_{j+1}} (t_m-t_j)^{\alpha-1}\mathcal{S}_{2h}(t_m-t_j)P_hF(X^h_j)ds\nonumber\\
	&+&\sum_{j=0}^{m-1}\int_{t_j}^{t_{j+1}}\mathcal{S}_{1h}(t_m-t_j)P_hG(X^h_j)dW(s) \nonumber\\
	&+&\int_0^{t_m}\mathcal{S}_{1h}(t_m-\llcorner s\lrcorner)P_h\Phi dB^H(s),
\end{eqnarray}
where the notation $\llcorner s\lrcorner$ is defined as in \cite[(89)]{Nou} by 
\begin{eqnarray}
	\label{not}
	\llcorner s\lrcorner:=\left[\frac{t}{\Delta t}\right] \Delta t.
\end{eqnarray}
Subtracting \eqref{itdismild} and \eqref{itnumexpsol}, applying triangle inequality, taking the square and the estimate $(a+b+c)^2\leq 3(a^2+b^2+c^2)$, it follows that 
\begin{eqnarray}
	\label{expcon*}
	&&err_1^2\nonumber\\
	&\leq& 3\left\|\sum_{j=0}^{m-1}\int_{t_j}^{t_{j+1}} \left((t_m-s)^{\alpha-1}\mathcal{S}_{2h}(t_m-s)P_hF(X^h(s))-\right.\right.\nonumber\\
	&&\hspace{3cm}\left.\left.
	(t_m-t_j)^{\alpha-1}\mathcal{S}_{2h}(t_m-t_j)P_hF(X^h_j)\right)ds\right\|^2_{L^2(\Omega;\mathcal{H})} \nonumber\\
	&+&3\left\| \sum_{j=0}^{m-1}\int_{t_j}^{t_{j+1}}\left(\mathcal{S}_{1h}(t_m-s)P_hG(X^h(s))-\mathcal{S}_{1h}(t_m-t_j)P_hG(X^h_j)\right)dW(s)\right\|^2_{L^2(\Omega;\mathcal{H})} \nonumber\\
	&+&\left\| \int_0^{t_m}\left(\mathcal{S}_{1h}(t_m-s)-\mathcal{S}_{1h}(t_m-\llcorner s\lrcorner)\right)P_h\Phi dB^H(s)\right\|^2_{L^2(\Omega;\mathcal{H})}\nonumber\\
	&=:&  3(V_1^2+V_2^2+V_3^2). 
\end{eqnarray}
Following closely the work done in \cite[(76)-(80)]{Nouc}, we have
\begin{eqnarray}
	\label{expcon1*}
	V_1\leq C\Delta t^{\frac{\min(\alpha(2H+\beta-1),2-2\alpha)}2}+C\Delta t^{\alpha}\sum_{j=0}^{m-1}\left\|X^h(t_j)-X^h_j\right\|_{L^2(\Omega;\mathcal{H})}.
\end{eqnarray}
Hence
\begin{eqnarray}
	\label{expcon1}
	V_1^2\leq C\Delta t^{\min(\alpha(2H+\beta-1),2-2\alpha)}+C\Delta t^{2\alpha-1}\sum_{j=0}^{m-1}\left\|X^h(t_j)-X^h_j\right\|^2_{L^2(\Omega;\mathcal{H})}.
\end{eqnarray}
By adding and subtracting a term, using triangle inequality, we recast $V_2$ as follows
\begin{eqnarray}
	\label{expcon2*}
	V_2&:=&\left\| \sum_{j=0}^{m-1}\int_{t_j}^{t_{j+1}}\left(\mathcal{S}_{1h}(t_m-s)P_hG(X^h(s))-\mathcal{S}_{1h}(t_m-t_j)P_hG(X^h_j)\right)dW(s)\right\|_{L^2(\Omega;\mathcal{H})} \nonumber\\
	&\leq&\left\| \sum_{j=0}^{m-1}\int_{t_j}^{t_{j+1}}\left[\mathcal{S}_{1h}(t_m-s)-\mathcal{S}_{1h}(t_m-t_j)\right]P_hG(X^h(s))dW(s)\right\|_{L^2(\Omega;\mathcal{H})} \nonumber\\
	&+&\left\| \sum_{j=0}^{m-1}\int_{t_j}^{t_{j+1}}\mathcal{S}_{1h}(t_m-t_j)P_h\left[G(X^h(s))-G(X^h(t_j))\right]dW(s)\right\|_{L^2(\Omega;\mathcal{H})} \nonumber\\
	&+&\left\| \sum_{j=0}^{m-1}\int_{t_j}^{t_{j+1}}\mathcal{S}_{1h}(t_m-t_j)P_h(\Delta t)\left[G(X^h(t_j))-G(X^h_j)\right]dW(s)\right\|_{L^2(\Omega;\mathcal{H})} \nonumber\\
	&=:& \sum_{i=1}^3 V_{2i}.
\end{eqnarray}
Using the martingale property of the stochastic integral, the It\^o isometry \eqref{browint}, inserting an appropriate power of $A_h$, \eqref{prepaexp21}, the semidiscrete version of \eqref{semigroup_prp} and \eqref{semigroup_prp*} with $\sigma=\eta=\delta=\frac{2H+\beta-1}{2}$ and $t_1=t_m-s$, $t_2=t_m-t_j$, Assumption \ref{diff} with $\tau=\frac{2H+\beta-1}2$ and the semidiscrete version of Theorem \ref{regfrac}(more precisely \eqref{sparegfrac}) yields

\begin{eqnarray}
	\label{expcon21}
	V_{21}^2&:=& \left\| \sum_{j=0}^{m-1}\int_{t_j}^{t_{j+1}}\left[\mathcal{S}_{1h}(t_m-s)-\mathcal{S}_{1h}(t_m-t_j)\right]P_hG(X^h(s))dW(s)\right\|^2_{L^2(\Omega;\mathcal{H})} \nonumber\\
	&=& \sum_{j=0}^{m-1}\mathbb{E}\left[\int_{t_j}^{t_{j+1}}\left\|\left[\mathcal{S}_{1h}(t_m-s)-\mathcal{S}_{1h}(t_m-t_j)\right]A_h^{-\frac{2H+\beta-1}2}A_h^{\frac{2H+\beta-1}2}P_hG(X^h(s))\right\|^2_{L^0_2}ds\right] \nonumber\\
	&\leq& \sum_{j=0}^{m-1}\mathbb{E}\left[\int_{t_j}^{t_{j+1}} \left\| A_h^{-\frac{2H+\beta-1}2}\left[\mathcal{S}_{1h}(t_m-s)-\mathcal{S}_{1h}(t_m-t_j)\right]\right\|^2_{L(\mathcal{H})}\left\| A_h^{\frac{2H+\beta-1}2}P_hG(X^h(s))\right\|^2_{L^0_2}ds\right] \nonumber\\
	&\leq& C\sum_{j=0}^{m-1}\mathbb{E}\left[\int_{t_j}^{t_{j+1}} (s-t_j)^{\alpha(2H+\beta-1)} \left\| A^{\frac{2H+\beta-1}2}G(X^h(s))\right\|^2_{L^0_2}ds\right] \nonumber\\
	&\leq& C\Delta t^{\alpha(2H+\beta-1)}\sum_{j=0}^{m-1}\int_{t_j}^{t_{j+1}}  \left(1+\mathbb{E}\left[\left\| A^{\frac{2H+\beta-1}2}X^h(s)\right\|^2\right]\right)ds \nonumber\\
	&\leq& C\Delta t^{\alpha(2H+\beta-1)}\int_{0}^{t_m}  \left(1+\sup_{0\leq s\leq T}\left\| A_h^{\frac{2H+\beta-1}2}X^h(s)\right\|^2_{L^2(\Omega;\mathcal{H})}\right)ds \nonumber\\
	&\leq& C\Delta t^{\alpha(2H+\beta-1)}.
\end{eqnarray} 
To estimate the last two terms $V_{22}^2$ and $V_{23}^2$, using the martingale property of the stochastic integral, applying It\^o isometry, boundedness of $P_h$ and $\mathcal{S}_{1h}$, Assumption \ref{diff} and the semidiscrete version of Theorem \ref{regfrac} (more precisely \eqref{timeregfrac}) holds
\begin{eqnarray}
	\label{expcon22}
	V_{22}^2&=:&\left\| \sum_{j=0}^{m-1}\int_{t_j}^{t_{j+1}}\mathcal{S}_{1h}(t_m-t_j)P_h\left[G(X^h(s))-G(X^h(t_j))\right]dW(s)\right\|^2_{L^2(\Omega;\mathcal{H})} \nonumber\\
	&=& \sum_{j=0}^{m-1}\mathbb{E}\left[\int_{t_j}^{t_{j+1}}\left\|\mathcal{S}_{1h}(t_m-t_j)P_h\left[G(X^h(s))-G(X^h(t_j))\right]\right\|^2_{L^0_2}ds\right] \nonumber\\
	&\leq& C\sum_{j=0}^{m-1}\left\|\mathcal{S}_{1h}(t_m-t_j)P_h\right\|^2_{L(\mathcal{H})}\int_{t_j}^{t_{j+1}}\mathbb{E}\left\|\left(G(X^h(s))-G(X^h(t_j))\right)\right\|^2_{L^0_2}ds \nonumber\\
	&\leq& C\sum_{j=0}^{m-1}\int_{t_k}^{t_{k+1}}\left\|X^h(s)-X^h(t_k)\right\|^2_{L^2(\Omega;\mathcal{H})}ds \nonumber\\
	&\leq& C\sum_{j=0}^{m-1}\int_{t_j}^{t_{j+1}}(s-t_j)^{\min(\alpha(2H+\beta-1),2-2\alpha)}ds \nonumber\\
	&\leq& C\Delta t^{\min(\alpha(2H+\beta-1),2-2\alpha)},
\end{eqnarray}
and 
\begin{eqnarray}
	\label{expcon23}
	V_{23}^2&=:&\left\| \sum_{j=0}^{m-1}\int_{t_j}^{t_{j+1}}\mathcal{S}_{1h}(t_m-t_j)\left[P_hG(X^h(t_j))-P_hG(X^h_j)\right]dW(s)\right\|^2_{L^2(\Omega;\mathcal{H})} \nonumber\\
	&=& \sum_{j=0}^{m-1}\mathbb{E}\left[\int_{t_j}^{t_{j+1}}\left\|\mathcal{S}_{1h}(t_m-t_j)P_h\left(G(X^h(t_j))-G(X^h_j)\right)\right\|^2_{L^0_2}ds\right] \nonumber\\
	&\leq&
	\sum_{j=0}^{m-1}\mathbb{E}\left[\int_{t_j}^{t_{j+1}}\left\|\mathcal{S}_{1h}(t_m-t_j)P_h\right\|^2_{L(\mathcal{H})}\left\|\left(G(X^h(t_j))-G(X^h_j)\right)\right\|^2_{L^0_2}ds\right] \nonumber\\
	&\leq& C\sum_{j=0}^{m-1}\int_{t_j}^{t_{j+1}}\mathbb{E}\left\|X^h(t_j)-X^h_j\right\|^2ds \nonumber\\
	&\leq& C\Delta t\sum_{j=0}^{m-1}\left\|X^h(t_j)-X^h_j\right\|^2_{L^2(\Omega;\mathcal{H})}.
\end{eqnarray}
Putting \eqref{expcon21}, \eqref{expcon22} 	and \eqref{expcon23} in \eqref{expcon2*} leads
\begin{eqnarray}
	\label{expcon2}
	V_2^2\leq C\Delta t^{{\min(\alpha(2H+\beta-1),2-2\alpha)}}+C\Delta t\sum_{k=0}^{m-1}\left\|X^h(t_k)-X^h_k\right\|^2_{L^2(\Omega;\mathcal{H})}.
\end{eqnarray}
Concerning the approximation of $V_3^2$, using \eqref{fracint}, inserting an appropriate power of $A_h$, the semidiscrete version of \eqref{semigroup_prp1*} and \eqref{semigroup_prp*} with $t_1=t_m-s$, $t_2=t_m-t_j$, $\kappa=\delta=\frac{1-\beta}2$ and $\varpi=H$, Lemma \ref{prepaexp2} with  $\sigma=\frac{\beta-1}2$ and Assumption \ref{fracdiff} yields
\begin{eqnarray}
	\label{expcon3}
	&&V_3^2\nonumber\\
	&:=&\left\| \int_0^{t_m}\left(\mathcal{S}_{1h}(t_m-s)-\mathcal{S}_{1h}(t_m-\llcorner s\lrcorner)\right)P_h\Phi dB^H(s)\right\|^2_{L^2(\Omega;\mathcal{H})}\nonumber\\
	&\leq& C(H)\sum_{i=0}^{\infty}\left(\int_0^{t_m}\left\|\left(\mathcal{S}_{1h}(t_m-s)-\mathcal{S}_{1h}(t_m-\llcorner s\lrcorner)\right)A_h^{\frac{1-\beta}2}A_h^{\frac{\beta-1}2}P_h\Phi Q^{\frac 12}e_i\right\|^{\frac 1H} ds\right)^{2H}\nonumber\\
	&\leq& C(H)\sum_{i=0}^{\infty}\left(\sum_{j=0}^{m-1}\int_{t_j}^{t_{j+1}}\left\|A_h^{\frac{1-\beta}2}\left(\mathcal{S}_{1h}(t_m-s)-\mathcal{S}_{1h}(t_m-t_j)\right)\right\|^{\frac 1H}_{L(\mathcal{H})}\right.\nonumber\\
	&&\hspace{3cm}\left.\left\|A_h^{\frac{\beta-1}2}P_h\Phi Q^{\frac 12}e_i\right\|^{\frac 1H} ds\right)^{2H}\nonumber\\
	&\leq& C\sum_{i=0}^{\infty}\left(\sum_{j=0}^{m-1}\int_{t_j}^{t_{j+1}}(s-t_j)^{\frac{\alpha(2H+\beta-1)}{2H}}(t_m-s)^{-\alpha}\left\|A^{\frac{\beta-1}2}\Phi Q^{\frac 12}e_i\right\|^{\frac 1H} ds\right)^{2H}\nonumber\\
	&\leq& C\Delta t^{\alpha(2H+\beta-1)}\left(\int_{0}^{t_m}(t_m-s)^{-\alpha} ds\right)^{2H}\left\|A^{\frac{\beta-1}2}\Phi\right\|^2_{L^0_2}\nonumber\\
	&\leq& C\Delta t^{\alpha(2H+\beta-1)}.
\end{eqnarray} 
Substituting \eqref{expcon1}, \eqref{expcon2} and \eqref{expcon3} in \eqref{expcon*} yields
\begin{eqnarray}
	\label{expcon**}
	&&\left\|X^h(t_m)-X^h_m\right\|^2_{L(\mathcal{H})}\nonumber\\
	&\leq&C\Delta t^{\min(\alpha(2H+\beta-1),2-2\alpha)}+C\Delta t^{2\alpha-1}\sum_{k=0}^{m-1}\left\|X^h(t_k)-X^h_k\right\|^2_{L(\mathcal{H})}.
\end{eqnarray}
Applying the discrete Gronwall's inequality to \eqref{expcon**} and taking the squared-root leads
\begin{eqnarray}
	\label{expcon}
	\left\|X^h(t_m)-X^h_m\right\|_{L(\mathcal{H})}&\leq&C\Delta t^{\frac{\min(\alpha(2H+\beta-1),2-2\alpha)}2}.
\end{eqnarray}
Adding \eqref{spacon} and \eqref{expcon} completes the proof.$\hfill\square$

\end{document}